\documentclass{imsart}

\RequirePackage{amsthm,amsmath,amsfonts,amssymb,maruyama,bm,mathtools}
\RequirePackage[authoryear]{natbib}
\RequirePackage[colorlinks,citecolor=blue,urlcolor=blue]{hyperref}
\RequirePackage{graphicx}

\startlocaldefs
\numberwithin{equation}{section}
\theoremstyle{plain}
\newtheorem{thm}{Theorem}[section]
\newtheorem{lemma}{Lemma}[section]

\theoremstyle{remark}
\newtheorem{remark}{Remark}[section]

\endlocaldefs

\begin{document}

\begin{frontmatter}
\title{On admissible estimation of a mean vector when the scale is unknown}
\runtitle{Admissibility}

%
%
%

\begin{aug}
\author[A]{
\fnms{Yuzo} \snm{Maruyama}\ead[label=e1]{maruyama@port.kobe-u.ac.jp}},
\and
\author[B]{
\fnms{William, E.} \snm{Strawderman}\ead[label=e2]{straw@stat.rutgers.edu}}
\address[A]{Graduate School of Business Administration,
Kobe University \\ \printead{e1}}

\address[B]{Department of Statistics and Biostatistics,
Rutgers University \\
\printead{e2}}
\runauthor{Y.~Maruyama and W.~E.~Strawderman}
\end{aug}

\begin{abstract}
We consider admissibility of generalized Bayes estimators of the mean of a multivariate normal distribution when the scale is unknown under quadratic loss. 
The priors considered put the improper invariant prior on the scale while the prior on the mean has a hierarchical normal structure conditional on the scale. 
This conditional hierarchical prior is essentially that of Maruyama and Strawderman (2021, Biometrika) (MS21)
which is indexed by a hyperparameter $a$. 
In that paper $a$ is chosen so this conditional prior is proper which corresponds to $a>-1$. 
This paper extends MS21 by considering improper conditional priors 
with $a$ in the closed interval $[-2, -1]$, and establishing admissibility for such $a$. 
The authors, in Maruyama and Strawderman (2017, JMVA), have earlier shown that such conditional priors with $a < -2$ lead to inadmissible estimators.
This paper therefore completes the determination of admissibility/inadmissibility for this class of priors. 
It establishes the the boundary as $a = -2$, with admissibility holding for $a\geq -2$ and inadmissibility for $a < -2$. 
This boundary corresponds exactly to that in the known scale case for these conditional priors, and which follows from Brown (1971, AOMS). 
As a notable benefit of this enlargement of the class of admissible generalized Bayes estimators, we give admissible and minimax estimators in all dimensions greater than $2$ as opposed to MS21 which required the dimension to be greater than $4$.
In one particularly interesting special case, we establish that the joint Stein prior for the unknown scale case leads to a minimax admissible estimator in all dimensions greater than $2$.
\end{abstract}


\end{frontmatter}

\section{Introduction}
We consider admissibility of generalized Bayes estimators of the mean of a multivariate normal distribution when the scale is unknown under quadratic loss.
Specifically, we consider the model $X\sim N_p(\theta,\sigma^2I)$, $S\sim \sigma^2\chi_n^2$
where $X$ and $S$ are independent with densities
\begin{equation}\label{our.density}
 \begin{split}
f(x\mid\theta,\eta)&=\frac{\eta^{p/2}}{(2\pi)^{p/2}}\exp\left(-\frac{\eta\|x-\theta\|^2}{2}\right), \\
f(s\mid\eta)&=\frac{\eta^{n/2}s^{n/2-1}}{\Gamma(n/2)2^{n/2}}\exp\left(-\frac{\eta s}{2}\right),
\end{split}
\end{equation}
with $\eta=1/\sigma^2$. 
The loss function is scaled quadratic loss 
\begin{equation}\label{loss}
L(\delta;\theta,\eta)= \eta\|\delta(x,s)-\theta\|^2.
\end{equation}
The estimators considered are generalized Bayes with respect to the generalized priors on $(\theta,\eta)$,
\begin{equation}\label{Prior}
\pi_*(\theta,\eta)=\frac{1}{\eta}\times\int\frac{\eta^{p/2}}{(2\pi)^{p/2}g^{p/2}}\exp\left(-\frac{\eta}{2g}\|\theta\|^2\right) \pi(g)\rd g,
\end{equation}
with the following hierarchical structure
\begin{equation}\label{Prior.1}
 \theta\mid \{g,\eta\} \sim N_p(0,(g/\eta) I),\quad \pi(g)= \frac{1}{(g+1)^{a+2}}\left(\frac{g}{g+1}\right)^b, \quad \eta \sim  \frac{1}{\eta}. 
\end{equation}
As detailed 
in Section \ref{sec:expression}, 
the generalized Bayes estimator under this prior is given by
\begin{equation}\label{delta_*_intro}
 \delta_*=\left(1-\frac
{\int_0^\infty (g+1)^{-p/2-1}(1+\|x\|^2/\{s(g+1)\})^{-p/2-n/2-1}\pi(g)\rd g}
{\int_0^\infty (g+1)^{-p/2}(1+\|x\|^2/\{s(g+1)\})^{-p/2-n/2-1}\pi(g)\rd g}
\right)x
\end{equation}
which is well-defined if
\begin{equation*}
 \int_0^\infty\frac{\pi(g)\rd g}{(g+1)^{p/2}}=\int_0^\infty
\left(\frac{g}{g+1}\right)^b\frac{\rd g}{(g+1)^{p/2+a+2}}<\infty.
\end{equation*}
This is assured provided
\begin{equation}
 p/2+a+1>0 \text{ and } b+1>0
\end{equation}
which we assume throughout this paper.

Admissibility of the estimator \eqref{delta_*_intro} is established in \cite{Maruyama-Strawderman-2020-arxiv} for the case $-1< a  <n/2$ and $ b>-1$, i.e., in cases where 
$\pi(g)$ is proper ($ \int_0^\infty  \pi(g)\rd g<\infty$), even though,
since the prior on $\eta$ is proportional to $1/\eta$, the prior $\pi_*(\theta,\eta)$ is improper.
This paper considers the more challenging problem where $\pi(g)$ is itself improper
($ \int_0^\infty  \pi(g)\rd g=\infty$).
We establish admissibility in two such cases:
\begin{align}
&\text{\textbf{CASE I} } \ \max(-p/2-1,-2)<a\leq -1 \text{ and }b>-1 \text{ under }p\geq 1,\label{our.a.b} \\
&\text{\textbf{CASE II} } \ a=-2\text{ and }b\geq 0 \text{ under } p\geq 3.\label{our.A.B}
\end{align}
The main theorem of this paper is as follows.
\begin{thm}\label{thm.main.intro}
Under either assumption, \eqref{our.a.b} or \eqref{our.A.B}, the generalized Bayes
estimator under $\pi_*$ is admissible among all estimators.
\end{thm}
The authors, in \cite{Maruyama-Strawderman-2017}, have earlier shown that such conditional priors with $a < -2$ lead to inadmissible estimators.
This paper therefore completes the determination of admissibility/inadmissibility for this class of priors. 
It establishes the the boundary as $a = -2$, with admissibility holding for $a\geq -2$ and inadmissibility for $a < -2$. 
The ultimate admissibility result by this paper 
result of this paper was foreshadowed by
\cite{Maruyama-Strawderman-2020}
where admissibility of generalized Bayes estimators with $a\geq -2$
was proved 
within the class of equivariant estimators of the form $\{1-\psi(\|x\|^2/s)\}x$.

The estimator $X$, with a constant risk $p$, is minimax for all $p$.
It is admissible for $p=1, 2$ whereas it is inadmissible for $p\geq 3$.
Therefore we are mainly interested in proposing admissible minimax estimators for $ p\geq 3$.
When $p\geq 3$, the generalized Bayes estimator has been shown to be minimax 
if $-p/2-1<a\leq \xi(p,n) $ and $b=0$ (\citeauthor{Lin-Tsai-1973}, \citeyear{Lin-Tsai-1973})
and if $-p/2-1<a\leq \xi(p,n) $ and $b>0$ (\citeauthor{Maru-Straw-2005}, \citeyear{Maru-Straw-2005}),
where
\begin{equation*}
\xi(p,n)=-2+\frac{(p-2)(n+2)}{2(2p+n-2)}.
\end{equation*}
Hence we have a following result.
\begin{thm}\label{thm.main.minimax.intro}
Assume $p\geq 3$ and $n\geq 2$. Then the generalized Bayes estimator under $\pi_*$ is minimax and admissible among the class of all estimators if
\begin{align*}
 -2\leq a\leq \xi(p,n)\text{ and }b\geq 0.
\end{align*}
\end{thm}
Figure \ref{fig1} presents a summary of admissibility/inadmissibility and minimaxity results 
for the unknown scale case for this class of priors.
\begin{remark}
\cite{Brown-1971} largely settled the issue of admissibility of generalized Bayes estimators of
$\mu$ in the known scale case, $X\sim N_p(\mu,I_p)$ with no $S$, under quadratic loss $\|d-\mu\|^2$.
Studies involving admissibility and minimaxity in the known scale case largely focused on priors with the hierarchical structure
\begin{equation}\label{prior.known}
\int\frac{1}{(2\pi)^{p/2}g^{p/2}}\exp\left(-\frac{1}{2g}\|\mu\|^2\right) \pi(g)\rd g
%
\end{equation}
with $\pi(g)$ given by \eqref{Prior.1} for $a>-p/2-1$ and $b>-1$. 
The key results under the prior above 
in the known scale case are summarized as follows and in Figure \ref{fig1}.
The estimator is inadmissible if $-p/2-1<a< -2$ and $b>-1$ (\citeauthor{Brown-1971}, \citeyear{Brown-1971}) and is admissible if $a\geq -2$ and $b>-1$ (\citeauthor{Brown-1971}, \citeyear{Brown-1971}). 
Furthermore the estimator is minimax if 
$ -1<a\leq p/2-3$ and $b=0$ (\citeauthor{Strawderman-1971}, \citeyear{Strawderman-1971}),
if $-p/2-1<a\leq -1 $ and $b=0$ (\citeauthor{Berger-1976}, \citeyear{Berger-1976})
and if $-p/2-1<a \leq p/2-3$ and $b>0$ (\citeauthor{Faith-1978}, \citeyear{Faith-1978}).
\end{remark}
\begin{figure}
\centering
\includegraphics[scale=0.67,trim=25 390 0 200]{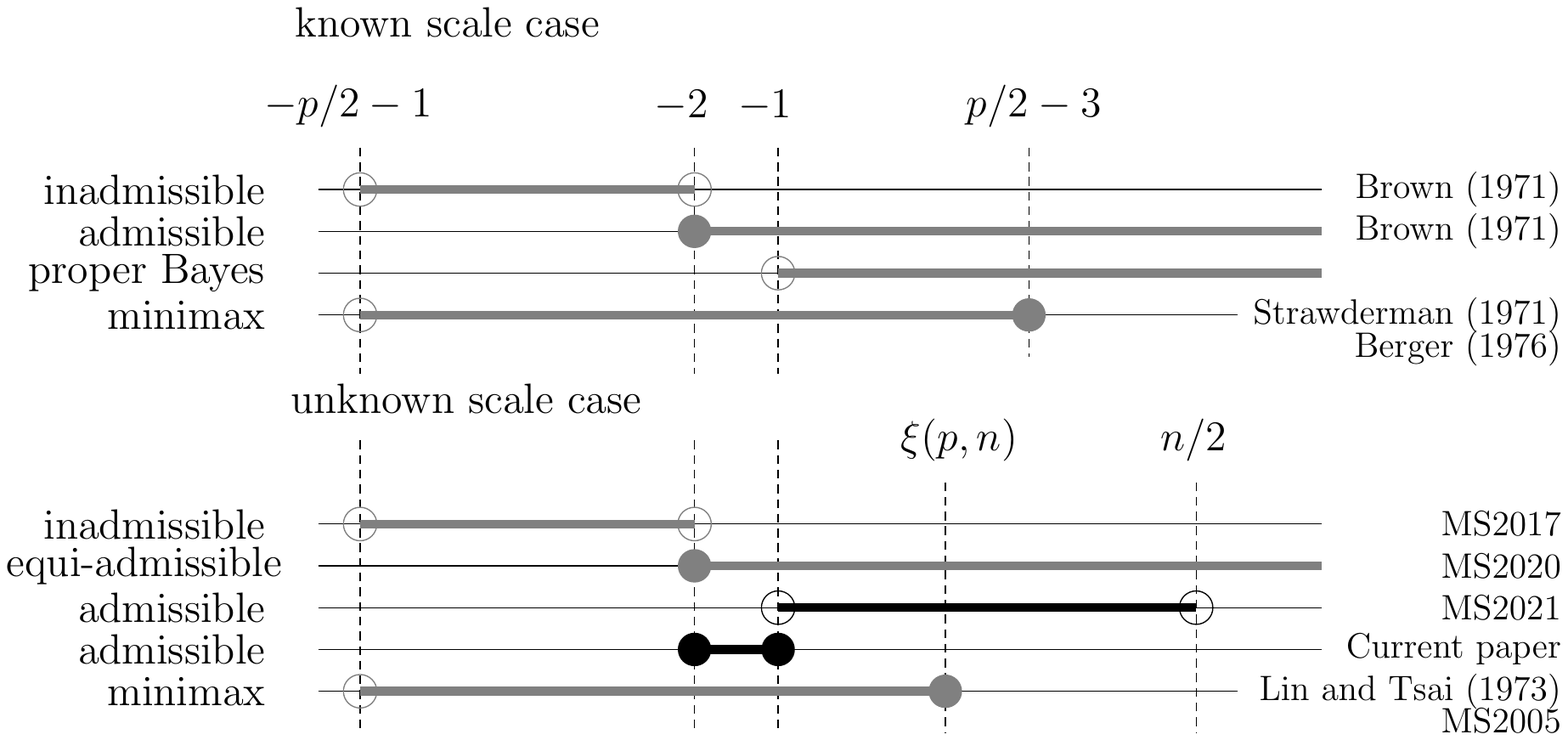}
\caption{Ranges of $a$ for admissibility/inadmissibility and minimaxity}
\label{fig1}
\end{figure}

We investigate admissibility using a version of Blyth's method closely related to the approach of \cite{Brown-Hwang-1982}.
To this end we construct a sequence of proper priors $\pi_{ij}(\theta,\eta)$
converging to $\pi_*$ of the form
\begin{equation}\label{our.seq}
\pi_{ij}(\theta,\eta)=\frac{h_i^2(\eta)}{\eta}\int\frac{\eta^{p/2}}{(2\pi)^{p/2}g^{p/2}}\exp\left(-\frac{\eta}{2g}\|\theta\|^2\right) \pi(g)k_j^2(g)\rd g,
\end{equation}
with the hierarchical structure
\begin{equation}\label{our.seq.1}
  \theta\mid \{g,\eta\} \sim N_p(0,(g/\eta) I),\quad g \sim \pi(g)k^2_j(g), 
\quad \eta \sim \frac{h_i^2(\eta)}{\eta}.
\end{equation}
In \eqref{our.seq} and \eqref{our.seq.1}, $h_i(\eta)$ and $k_j(g)$ are given as follows,
\begin{equation}\label{EQ:h_i}
 h_i(\eta)=\frac{i}{i+|\log\eta|},
\end{equation}
\begin{equation}\label{EQ:k_j}
 k_j( g )=1-\frac{\log( g +1)}{\log( g +1+j)}.
\end{equation}
Properties of $h_i(\eta)$ and $k_j(g)$ will be provided in Lemmas \ref{lem.hi} and \ref{lem:k_j}.
In particular, we emphasize that
$h_i^2(\eta)/\eta$ and $\pi(g)k^2_j( g )$ are both proper by
part \ref{lem.hi.02} of Lemma \ref{lem.hi} and part \ref{lem:k_j.2.5} of Lemma \ref{lem:k_j}, respectively.
Also we note that 
eventually, for large $i$ and $j$ we set
\begin{equation}\label{EQ:ilogj}
 i=\log(1+j)
\end{equation}
which is crucial in the proof of admissibility.

%

Before proceeding with the development of the main result, we make
several remarks relating the current paper to earlier developments.
\begin{remark}
\cite{Maruyama-Strawderman-2020-arxiv} considered the joint improper prior given by
\begin{align}\label{Prior.intro.0}
\underbrace{\frac{1}{\eta}}_{\text{improper}}
\times\underbrace{\int\frac{\eta^{p/2}}{(2\pi)^{p/2}g^{p/2}}\exp\left(-\frac{\eta}{2g}\|\theta\|^2\right) \pi(g)\rd g}_{\text{proper}}
\end{align}
where $\pi(g)$ is proper, or $a>-1$, 
and 
proposed a class of admissible generalized Bayes estimators for any dimension $p$. 
Since the Stein phenomenon occurs for $p\geq 3$, we are interested in proposing admissible minimax estimators for $p\geq 3$.
As in Figure \ref{fig1}, the intersection of the minimaxity region $-p/2-1<a\leq \xi(p,n) $ and the admissibility region $a>-1$
was non-empty when $n\geq 3$ and $ p>4n/(n-2)$. 
Hence in particular no admissible minimax estimator had been found for $p=3,4$.
Theorem \ref{thm.main.minimax.intro} provides such estimators in this case for $n\geq 2$ 
and hence help fill this void, albeit with priors which are improper in $\theta$
conditional on $\eta$, since the generalized prior on $g$ is improper in these cases. 
Thus, not only have we extended the class of admissible estimators in all dimensions, 
but we provide admissible minimax estimators for dimensions $p=3$ and $4$, 
where no such estimators were previously known.
\end{remark}

\begin{remark}\label{rem:interesting}
In Theorem \ref{thm.main.minimax.intro}, two interesting cases seem deserving of attention. When $a=-2$ and $b=0$,
the prior corresponds to the joint Stein prior 
\begin{align*}
 \frac{1}{\eta}\times \eta^{p/2}\left\{\eta\|\theta\|^2\right\}^{1-p/2}=\|\theta\|^{2-p}.
\end{align*}
That this estimator is admissible and minimax follows from Theorem \ref{thm.main.minimax.intro}.
Additionally, \cite{Kubokawa-1991} established that this estimator dominates
the \cite{James-Stein-1961} estimator
\begin{equation*}
 \left(1-\frac{(p-2)/(n+2)}{\|x\|^2/s}\right)x.
\end{equation*}
Another interesting case is a variant of the James-Stein of the simple form 
\begin{equation*}
 \left(1-\frac{(p-2)/(n+2)}{\|x\|^2/s+(p-2)/(n+2)+1}\right)x,
\end{equation*}
which is generalized Bayes corresponding to $a=-2$ and $b=n/2$, that is,
\begin{equation*}
\pi(g)=\left(\frac{g}{g+1}\right)^{n/2}.
\end{equation*}
This estimator is also admissible and minimax. See Section 4.1 of \cite{Maruyama-Strawderman-2020}
and Section 3 of \cite{Maruyama-Strawderman-2020-arxiv} for details.
\end{remark}


\begin{remark}
The proof of admissibility in this paper is closely related to and greatly influenced by those in \cite{Brown-Hwang-1982} (BH) and \cite{James-Stein-1961} (JS). 
It is also similar to that in \cite{Maruyama-Strawderman-2020-arxiv} (MS). 
It seems worthwhile to comment on some of the technical differences.
In MS the (conditional on $\eta$) priors on $\theta$ were proper and only the prior on $\eta$ was improper. 
Hence only a proper sequence of priors on the (inverse) scale, $\eta$, was required. 
In both BH and JS the variance is assumed known (for the normal case in BH) 
and hence the required sequence of proper priors is needed only on $\theta$, 
whereas we need to find proper sequences on both $\theta$, and the (inverse) scale, $\eta$.
In this paper, since the (conditional on $\eta$) priors on $\theta$ are also improper, 
the joint improper prior is given by
\begin{align}\label{Prior.intro.00}
\underbrace{\frac{1}{\eta}}_{\text{improper}}
\times\underbrace{\int\frac{\eta^{p/2}}{(2\pi)^{p/2}g^{p/2}}\exp\left(-\frac{\eta}{2g}\|\theta\|^2\right) \pi(g)\rd g}_{\text{improper}}
\end{align}
where $\pi(g)$ is improper. As we mentioned, for the invariant prior $ 1/\eta$,
we use the sequence of proper priors, $h_i^2(\eta)/\eta$, which was investigated in MS.

Our sequence of proper priors on $g$,
$\pi(g)k_j^2(g)$, with $k_j(g)$ given by \eqref{EQ:k_j}, may be viewed as a modified version of that in BH and in JS. 
In particular the sequence in BH is
\begin{align*}
 k^{\mathrm{BH}}_j(g)=
\begin{cases}
 1 & 0\leq g\leq 1 \\
 1-\log g/\log (j+1) & 1<g<j+1 \\
0 & g\geq j+1.
\end{cases}
\end{align*}
BH proposes a sufficient condition for admissibility of generalized Bayes estimator with two components, which they referred to as an ``asymptotic flatness'' condition and ``growth'' condition. 
Because our proof is similar in spirit to that of BH, there are some connections in terms of these two conditions between BH and this paper. 
We focus here on the ``growth'' condition. 
In section \ref{sec:CASEII}, the condition for integrability 
(in order to invoke the dominated convergence theorem) in \eqref{product.b} and \eqref{product.c}
\begin{align}
\sup_{i,j\in\mathbb{N}} \left\{\int \eta\{h'_i(\eta)\}^2\rd\eta \int \frac{\pi(g)k^2_j(g)}{g+1}\rd g\right\}<\infty \label{intro.growth.1}\\
\sup_{i,j\in\mathbb{N}}\left\{\int \frac{h_i^2(\eta)}{\eta}\rd\eta \int (g+1)\pi(g)\{k'_j(g)\}^2\rd g\right\} <\infty\label{intro.growth.2}
\end{align}
for all $i,j\in\mathbb{N}$ may be regarded as a ``growth condition''.
Of the four integrals appearing in \eqref{intro.growth.1} and \eqref{intro.growth.2}, only the integral
\begin{align*}
\int (g+1)\pi(g)\{k'_j(g)\}^2\rd g <\infty
\end{align*} 
appeared in \cite{Brown-Hwang-1982}. They bounded $(\rd/\rd g)k^{\mathrm{BH}}_j(g)$ as
\begin{align*}
\left| (\rd/\rd g)k^{\mathrm{BH}}_j(g)\right|\leq
\begin{cases}
 0 & g\leq 1 \\
1/(g\log 2) & 1<g\leq 2 \\
1/(g\log g) & g> 2,
\end{cases}
\end{align*}
for all $j\in\mathbb{N}$ and concluded that
\begin{align*}
\sup_{j\in\mathbb{N}}\int (g+1)\pi(g)\{(\rd/\rd g)k^{\mathrm{BH}}_j(g)\}^2\rd g <\infty.
\end{align*} 
However Lemmas \ref{lem.hi} gives
\begin{align*}
 \int \frac{h_i^2(\eta)}{\eta}\rd\eta=2i
\end{align*}
which does not imply the finiteness of \eqref{intro.growth.2}.

In this paper, we very carefully bound the integrals in \eqref{intro.growth.1} and 
\eqref{intro.growth.2} from above.
By Lemmas \ref{lem.hi} and \ref{lem:k_j} as well as $\pi(g)\leq 1$ under CASE II, we have
\begin{equation}\label{product.bc}
 \begin{split}
 \int \eta\{h'_i(\eta)\}^2\rd\eta \int \frac{\pi(g)k^2_j(g)}{g+1}\rd g< 4\frac{\log(1+j)}{i}, \\
\int \frac{h_i^2(\eta)}{\eta}\rd\eta \int (g+1)\pi(g)\{k'_j(g)\}^2\rd g <10\frac{i}{\log(1+j)}.
\end{split}
\end{equation}
With the choice $i=\log(1+j)$, the finiteness given by \eqref{intro.growth.1} and \eqref{intro.growth.2} follows.
\end{remark}

The organization of this paper is as follows. 
Section \ref{sec:expression} is denoted to developing expressions for Bayes estimators and the risk
differences which are used to prove Theorem \ref{thm.main.intro}.
Sections \ref{sec:CASEI} and \ref{sec:CASEII} are  devoted to the proof of Theorem \ref{thm.main.intro} for CASES I and II given by
\eqref{our.a.b} and \eqref{our.A.B}, respectively.
many of the proofs of technical lemmas are given in Appendix.

\section{The form of Bayes estimators and risk differences}
\label{sec:expression}
In this section we develop expressions for Bayes estimators and the risk
differences which are used to prove Theorem \ref{thm.main.intro}.
Let
\begin{equation}\label{MM}
m(\psi(\theta,\eta)) 
=\iint \psi(\theta,\eta)f(x\mid\theta,\eta)f(s\mid\eta)\rd \theta\rd \eta .
\end{equation}
Then, under the loss \eqref{loss},
the generalized Bayes estimator under the improper $\pi_*(\theta,\eta)$ is 
\begin{equation}\label{GB}
\delta_*= \frac{m(\eta\theta\pi_*(\theta,\eta))}{m(\eta\pi_*(\theta,\eta))},
\end{equation}
and the proper Bayes estimator under the proper $\pi_{ij}(\theta,\eta)$ is 
\begin{equation}\label{PB.ij}
\delta_{ij}= \frac{m(\eta\theta\pi_{ij}(\theta,\eta))}{m(\eta\pi_{ij}(\theta,\eta))}.
\end{equation}
The Bayes risk difference under $\pi_{ij}$ given by
\begin{align}\label{bill.0}
\Delta_{ij} =
\int_{\mbR^p}\int_0^\infty 
 \left\{E\left(\eta\|\delta_*-\theta\|^2\right)-E\left(\eta\|\delta_{ij}-\theta\|^2\right) \right\}
\pi_{ij}(\theta,\eta)
 \rd  \theta \rd \eta
\end{align}
may be re-expressed as
\begin{equation}\label{DELTA}
 \Delta_{ij}=\int_{\mbR^p}\int_0^\infty 
\| \delta_*-\delta_{ij}\|^2 m(\eta\pi_{ij}(\theta,\eta))\rd x\rd s.
\end{equation}

The basic structure of the proof is standard, as in \cite{Brown-Hwang-1982}, 
and is based on the method of \cite{Blyth-1951}.
The following form of Blyth's sufficient condition shows that 
$ \lim_{i,j\to\infty}\Delta_{ij} =0$ 
implies admissibility. 
\begin{lemma}\label{lem.Blyth}
Suppose $ \pi_{ij}(\theta,\eta)$ is an increasing (in $i$ and $j$) sequence of proper priors, 
$\lim_{i,j\to\infty}\pi_{ij}(\theta,\eta)=\pi(\theta,\eta)$ and 
$\pi_{ij}(\theta,\eta)>0$ for all $\theta$ and $\eta$. Then 
$\delta_*$ is admissible if $\Delta_{ij}$ satisfies 
\begin{equation}\label{bill.1}
 \lim_{i,j\to\infty}\Delta_{ij} =0.
\end{equation}
\end{lemma} 
Both 
Lemma \ref{lem.hi} and 
Lemma \ref{lem:k_j}
guarantee that $\pi_{ij}(\theta,\eta) $ given by \eqref{our.seq}
satisfies the assumptions of Lemma \ref{lem.Blyth} as follows.
\begin{lemma}\label{lem.hikj}
The prior $\pi_{ij}(\theta,\eta) $ given by \eqref{our.seq}
is increasing in $i$ and $j$ and integrable for all fixed $i$ and $j$. 
Further $\lim_{i,j\to\infty}\pi_{ij}(\theta,\eta)=\pi_*(\theta,\eta)$ and 
$\pi_{ij}(\theta,\eta)>0$ for all $\theta$ and $\eta$.
\end{lemma}
Now we rewrite $ \delta_*$, $ \delta_{ij}$ and the integrand of \eqref{DELTA}.
Using the identity, 
\begin{align*}
 \|x-\theta\|^2+\frac{\|\theta\|^2}{g}=\frac{g+1}{g}\left\|\theta-\frac{g}{g+1}x\right\|^2+\frac{\|x\|^2}{g+1},
\end{align*}
we have 
\begin{align}\label{MM.1}
m(\eta\pi_{ij}) 
&=\iint \{\eta \pi_{ij}(\theta,\eta)\}f(x\mid\theta,\eta)f(s\mid\eta)\rd \theta\rd \eta \\
&= \iiint \frac{\eta^{p/2}}{(2\pi)^{p/2}}\exp\left(-\eta\frac{\|x-\theta\|^2}{2}\right)
f(s\mid\eta) \notag\\
&\quad\times 
\frac{\eta^{p/2}}{(2\pi)^{p/2}g^{p/2}}\exp\left(-\frac{\eta}{2g}\|\theta\|^2\right) h_i^2(\eta)\pi(g)k_j^2(g)\rd \theta\rd g\rd \eta \notag\\
&=\iint \frac{\eta^{p/2}f(s\mid\eta)}{(2\pi)^{p/2}(g+1)^{p/2}}\exp\left(-\frac{\eta\|x\|^2}{2(g+1)}\right)
 h_i^2(\eta)\pi(g)k_j^2(g)\rd g\rd \eta \notag\\
&=q_1(p,n)s^{n/2-1} \iint F(g,\eta ;  w,s)
h_i^2(\eta)\pi(g)k_j^2(g)\rd g\rd \eta, \notag
\end{align}
where $w=\|x\|^2/s$,
\begin{equation}\label{EQ:F}
 F(g,\eta ;  w,s)=\frac{\eta^{p/2+n/2}}{(g+1)^{p/2}}\exp\left\{-\frac{\eta s}{2}
\left(\frac{w}{g+1}+1\right)\right\},
\end{equation}
 and
\begin{align*}
 q_1(p,n)= \frac{1}{(2\pi)^{p/2}\Gamma(n/2)2^{n/2}}.
\end{align*}
Similarly we have
\begin{equation}\label{MM.2}
 \begin{split}
 m(\eta\theta\pi_{ij}) 
&=\iint \{\eta \theta\pi_{ij}(\theta,\eta)\} f(x\mid\theta,\eta)f(s\mid\eta)\rd \theta\rd \eta \\
&=q_1(p,n) s^{n/2-1} \iint \frac{gx}{g+1}F(g,\eta ;  w,s)
h_i^2(\eta)\pi(g)k_j^2(g)\rd g\rd \eta.
\end{split}
\end{equation}
By \eqref{MM.1} and \eqref{MM.2}, the Bayes estimator under $\pi_{ij}$ is
\begin{equation}\label{EQ:Bayes_estimator}
\begin{split}
 \delta_{ij}
=\frac{m(\theta\eta\pi_{ij})}{m(\eta\pi_{ij})}
=\left(1-\frac{\phi_{ij}(w,s)}{w}\right)x,
\end{split}
\end{equation}
where 
\begin{align}\label{phi.ij.original}
 \phi_{ij}(w,s)&=w\frac{\iint (g+1)^{-1}F(g,\eta;w,s)h_i^2(\eta)\pi(g)k_j^2(g)\rd g\rd \eta}
{\iint  F(g,\eta;w,s)h_i^2(\eta)\pi(g)k_j^2(g)\rd g\rd \eta} .
\end{align}
With $h_i\equiv 1$ and $k_j\equiv 1$ in \eqref{phi.ij.original}, we have
\begin{align}\label{phi.*.original}
\phi_*(w,s) =w\frac{\iint (g+1)^{-1}F(g,\eta;w,s)\pi(g)\rd g\rd \eta}
{\iint  F(g,\eta;w,s)\pi(g)\rd g\rd \eta}
\end{align}
and our target generalized Bayes estimator given by
\begin{equation}\label{EQ:GBayes_estimator}
 \delta_*=\left(1-\frac{\phi_*(w,s)}{w}\right)x.
\end{equation}
Note that
\begin{align*}
 \int F(g,\eta;w,s)\rd \eta=\frac{\Gamma(p/2+n/2+1)}{(g+1)^{p/2}}\left(\frac{2}{s\{1+w/(g+1)\}}\right)^{p/2+n/2+1}
\end{align*}
which implies 
\begin{align}\label{phi**}
 \frac{\phi_*(w,s)}{w}=\frac{\int_0^\infty (g+1)^{-p/2-1}\{1+w/(g+1)\}^{-p/2-n/2-1}\pi(g)\rd g}
{\int_0^\infty (g+1)^{-p/2}\{1+w/(g+1)\}^{-p/2-n/2-1}\pi(g)\rd g}.
\end{align}
In the following, however, we keep \eqref{phi.*.original} not \eqref{phi**} as the expression of $\phi_*(w,s)$.

By \eqref{MM.1}, \eqref{EQ:Bayes_estimator} and \eqref{EQ:GBayes_estimator}, $ \left\| \delta_*-\delta_{ij} \right\|^2 m(\eta\pi_{ij})$, in $\Delta_{ij}$, is
\begin{equation}\label{EQ.integrand}
 \begin{split}
 \left\| \delta_*-\delta_{ij} \right\|^2 m(\eta\pi_{ij}) &= 
\|x\|^2 \left(\frac{\phi_*(w,s)}{w}-\frac{\phi_{ij}(w,s)}{w}\right)^2 m(\eta\pi_{ij}) \\
&=q_1(p,n)\|x\|^2 s^{n/2-1} A(\pi;i,j),
\end{split}
\end{equation}
where
 \begin{align}
&A(\pi;i,j)\label{Apiij}\\ &=\left(
\frac{\iint (g+1)^{-1}F \pi \rd g\rd \eta}
{\iint  F\pi \rd g\rd \eta}-
\frac{\iint (g+1)^{-1}F h_i^2 \pi k_j^2\rd g\rd \eta}
{\iint  F h_i^2 \pi k_j^2 \rd g\rd \eta}
\right)^2\iint  F h_i^2 \pi k_j^2 \rd g\rd \eta.\notag
\end{align}
In Sections \ref{sec:CASEI} and \ref{sec:CASEII} respectively, 
we will complete the proof of Theorem \ref{thm.main.intro}, 
for CASE I and CASE II by the dominated convergence theorem.
We do so by showing the integrand in
$\Delta_{ij}=q_1(p,n)\iint \|x\|^2 s^{n/2-1} A(\pi;i,j)\rd x\rd s$
is bounded by an integrable function. 
It follows, since the integrand approaches $0$, that $ \Delta_{ij}\to 0$
which establishes the result.
\section{CASE I}
\label{sec:CASEI}
This section is devoted to the proof of Theorem \ref{thm.main.intro} for CASE I given by
\eqref{our.a.b}.
Applying
the inequality
\begin{equation}\label{eq.4}
\left(\sum_{i=1}^k a_i\right)^2\leq k \sum_{i=1}^k a_i^2,
\end{equation}
to \eqref{Apiij}, we have
\begin{align}\label{A.inequality}
&A(\pi;i,j) \\
&\leq 3\left\{\left(
\frac{\iint \psi(g) F\pi\rd g\rd \eta}{\iint  F\pi\rd g\rd \eta}
-\frac{\iint \psi(g) Fh_i^2\pi\rd g\rd \eta}{\iint  Fh_i^2\pi\rd g\rd \eta}
\right)^2 \iint  F h_i^2\pi k_j^2\rd g\rd \eta\right. \notag\\
& \quad\left(
\frac{\iint \psi(g) Fh_i^2\pi\rd g\rd \eta}{\iint  Fh_i^2\pi\rd g\rd \eta}
-
\frac{\iint \psi(g) Fh_i^2\pi k_j^2\rd g\rd \eta}{\iint  Fh_i^2\pi\rd g\rd \eta}
\right)^2 
\iint  F h_i^2\pi k_j^2\rd g\rd \eta \notag\\
&\quad\left.\left(
\frac{\iint \psi(g) Fh_i^2\pi k_j^2\rd g\rd \eta}{\iint  Fh_i^2\pi\rd g\rd \eta}
-\frac{\iint \psi(g) Fh_i^2\pi k_j^2\rd g\rd \eta}{\iint  Fh_i^2\pi k_j^2\rd g\rd \eta}
\right)^2 
\iint  F h_i^2\pi k_j^2\rd g\rd \eta \right\} \notag\\
&\leq 3\left\{\mcA_1(\psi;\pi;i,j)+\mcA_2(\psi;\pi;i,j)+\mcA_3(\psi;\pi;i,j)\right\}\notag
\end{align}
where $\psi(g)=1/(g+1)$ and
\begin{align*}
\mcA_1(\psi)&= 
\left\{
\iint \psi(g)\left|\frac{1}{\iint  F\pi\rd g\rd \eta}-\frac{h_i^2}{\iint  Fh_i^2\pi\rd g\rd \eta}\right|F\pi\rd g\rd \eta \right\}^2 
\iint  F h_i^2\pi \rd g\rd \eta, \\
\mcA_2(\psi) &=
\frac{\dps \left(\iint \psi(g) F h_i^2 \pi (1-k_j^2)\rd g\rd \eta\right)^2}{\iint  Fh_i^2\pi\rd g\rd \eta}, \\
\mcA_3(\psi) &=
\frac{\left(\iint \psi(g) Fh_i^2\pi k_j^2\rd g\rd \eta\right)^2}
{(\iint  Fh_i^2\pi\rd g\rd \eta)^2 \iint  Fh_i^2\pi k_j^2\rd g\rd \eta} 
\left( \iint F h_i^2 \pi \rd g\rd \eta - \iint F h_i^2\pi k^2_j \rd g\rd \eta\right)^2.
\end{align*}
Since $\mcA_i$ for $i=1,2,3$ with some $\psi$ will also appear in Section \ref{sec:CASEII} (more precisely in Lemma \ref{thm:bcd}),
we summarize useful properties in the following Lemma.
\begin{lemma}\label{thm:a1a2a3}
Let
\begin{equation}
 \psi(g)=\frac{1}{(g+1)^{\alpha}} \text{ for }\alpha=1\text{ or }2. 
\end{equation}
Let $m$ satisfy $a+m+2\alpha>0$.
Also let $\epsilon$ satisfy
\begin{align*}
0< \epsilon<\frac{1}{2}\min\left(p/2+a+1,a+m+2\alpha,1\right).
\end{align*}
Then there exist positive constants, $\tilde{\mcA}_1$, $\tilde{\mcA}_2$, and 
$\tilde{\mcA}_3$, independent of $i$ and $j$, such that
\begin{align}
& \iint \frac{\|x\|^2s^{n/2-1}}{(\|x\|^2/s)^m} \mcA_1(\psi;\pi;i,j)\rd x\rd s \leq 
\tilde{\mcA}_1
B(a+m+2\alpha,b+1), \\
& \iint \frac{\|x\|^2s^{n/2-1}}{(\|x\|^2/s)^m} \mcA_\ell(\psi;\pi;i,j)\rd x\rd s \notag\\
&\quad\leq \frac{i}{\{\log(1+j)\}^2}
\frac{\tilde{\mcA}_\ell}{\epsilon}B(a+m+2\alpha-2\epsilon,b+1) \text{ for }\ell=2,3. 
\end{align}
\end{lemma}
\begin{proof}
See Appendix \ref{sec:thm:a1a2a3}.
\end{proof}
We set $m=0$ and $\alpha=1$ in Lemma \ref{thm:a1a2a3}. With the choice
\begin{equation}
 \epsilon=\frac{1}{4}\min\left(p/2+a+1,a+2,1\right), \ i=O(\{\log(1+j)\}^2),
\end{equation}
we have
\begin{align}
\sup_{\substack{j\in \mathbb{N} \\ i=O(\{\log(1+j)\}^2)}}\sum_{\ell=1}^3\iint \|x\|^2s^{n/2-1} \mcA_\ell(1/(g+1);\pi;i,j)\rd x\rd s <\infty.
\end{align}
By the dominated convergence theorem, we have a following result.
\begin{thm}\label{thm:main.case.1}
For $p\geq 1$, the generalized Bayes estimator under $\pi_*$ with 
\begin{align*}
\max(-p/2-1,-2)<a\leq -1 \text{ and }b>-1
\end{align*}
is admissible among the class of all estimators.
\end{thm}

\section{CASE II}
\label{sec:CASEII}
This section is devoted to the proof of Theorem \ref{thm.main.intro} for CASE II given by
\eqref{our.A.B}, $a=-2$, $b\geq 0$ and $p\geq 3$.
We need the condition $b\geq 0$ for $\pi(0)<\infty$, which is required in the following lemma.

\begin{lemma}\label{LEM:integral.parts.a=2} 
Assume $ a=-2$ and $b\geq 0$. Then
\begin{align}
& (n/2+1)\phi_{ij}(w,s) \label{phi.ij}\\ 
&=(p/2-1) -2w\frac{\iint (g+1)^{-1}\eta F(g,\eta;w,s)h_i(\eta)h'_i(\eta)\pi(g)k_j^2(g)\rd g\rd \eta}{\iint  F(g,\eta;w,s)h_i^2(\eta)\pi(g)k_j^2(g)\rd g\rd \eta} \notag \\ 
&\quad -2\frac{\iint (g+1+w)F(g,\eta;w,s)h_i^2(\eta)\pi(g)k_j(g)k'_j(g)\rd g\rd \eta}{\iint  F(g,\eta;w,s)h_i^2(\eta)\pi(g)k_j^2(g)\rd g\rd \eta} - \varphi_{ij}(w,s),\notag
\end{align}
where
\begin{align*}
 \varphi_{ij}(w,s)=
\begin{cases}
\displaystyle b \frac{\iint \{(g+1+w)/g(g+1)\}F(g,\eta;w,s)h_i^2(\eta)\pi(g)k_j^2(g)\rd g\rd \eta}{\iint  F(g,\eta;w,s)h_i^2(\eta)\pi(g)k_j^2(g)\rd g\rd \eta} &b>0 \\
\displaystyle(1+w)\frac{\int F(0,\eta;w,s)h_i^2(\eta)\rd \eta}{\iint  F(g,\eta;w,s)h_i^2(\eta)k_j^2(g)\rd g\rd \eta} & b=0. 
\end{cases}
\end{align*}
\end{lemma}
\begin{proof}
 See Appendix \ref{sec:LEM:integral.parts.a=2}.
\end{proof}
With $h_i\equiv 1$ and $k_j\equiv 1$ in \eqref{phi.ij}, we have
\begin{align}\label{phi.*}
 (n/2+1)\phi_*(w,s) =(p/2-1) -\varphi_*(w,s),
\end{align}
where 
\begin{equation}\label{varphi*}
 \varphi_*(w,s)=
\begin{cases}
\displaystyle b \frac{\iint \{(g+1+w)/g(g+1)\}F(g,\eta;w,s)\pi(g)\rd g\rd \eta}
{\iint  F(g,\eta;w,s)\pi(g)\rd g\rd \eta} &b>0 \\
\displaystyle(1+w)\frac{\int F(0,\eta;w,s)\rd \eta}{\iint  F(g,\eta;w,s)\rd g\rd \eta} & b=0. 
\end{cases}
\end{equation}
%
By \eqref{EQ.integrand}, \eqref{phi.ij} and \eqref{phi.*}, 
$ \left\| \delta_*-\delta_{ij} \right\|^2 m(\eta\pi_{ij})$, in $\Delta_{ij}$, is rewritten as
 \begin{align}
 \left\| \delta_*-\delta_{ij} \right\|^2 m(\eta\pi_{ij}) &= 
\|x\|^2 \left(\frac{\phi_*(w,s)}{w}-\frac{\phi_{ij}(w,s)}{w}\right)^2 m(\eta\pi_{ij}) \label{phi.integ.parts}\\
&=\frac{q_1(p,n)\|x\|^2s^{n/2-1}}{(n/2+1)^2 w^2}
\biggl\{\varphi_*(w,s)-\varphi_{ij}(w,s)  \notag\\ 
&\quad -2w\frac{\iint (g+1)^{-1}\eta F(g,\eta;w,s)h_i(\eta)h'_i(\eta)\pi(g)k_j^2(g)\rd g\rd \eta}{\iint  F(g,\eta;w,s)h_i^2(\eta)\pi(g)k_j^2(g)\rd g\rd \eta} \notag \\ 
&\quad  -2\frac{\iint (g+1+w)F(g,\eta;w,s)h_i^2(\eta)\pi(g)k_j(g)k'_j(g)\rd g\rd \eta}{\iint  F(g,\eta;w,s)h_i^2(\eta)\pi(g)k_j^2(g)\rd g\rd \eta} \biggr\}^2 \notag\\
&\quad \times \iint  F(g,\eta;w,s)h_i^2(\eta)\pi(g)k_j^2(g)\rd g\rd \eta.\notag
\end{align}
Applying the inequality \eqref{eq.4} to \eqref{phi.integ.parts}, we have
\begin{equation}
\begin{split}
 & \left\| \delta_*-\delta_{ij} \right\|^2 m(\eta\pi_{ij}) \\  
& \leq \frac{3q_1(p,n)}{(n/2+1)^2}\|x\|^2s^{n/2-1}
\left\{4\mathcal{B}(\pi;i,j)+4\mathcal{C}(\pi;i,j)+\mathcal{D}(\pi;i,j) \right\},
\end{split}
\end{equation}
where
\begin{align*}
\mathcal{B}(\pi;i,j)&=
\frac{\{\iint (g+1)^{-1}\eta F(g,\eta;w,s)h_i(\eta)h'_i(\eta)\pi(g)k_j^2(g)\rd g\rd \eta\}^2}{\iint  F(g,\eta;w,s)h_i^2(\eta)\pi(g)k_j^2(g)\rd g\rd \eta},\\
\mathcal{C}(\pi;i,j)&=
\frac{\{\iint (g+1+w)F(g,\eta;w,s)h_i^2(\eta)\pi(g)k_j(g)k'_j(g)\rd g\rd \eta\}^2}{w^2\iint  F(g,\eta;w,s)h_i^2(\eta)\pi(g)k_j^2(g)\rd g\rd \eta}, \\
 \mathcal{D}(\pi;i,j)&=\frac{\left\{\varphi_*(w,s)-\varphi_{ij}(w,s)\right\}^2}{w^2}
\iint  F(g,\eta;w,s)h_i^2(\eta)\pi(g)k_j^2(g)\rd g\rd \eta.
\end{align*}
For $\mathcal{B}$, by Cauchy-Schwarz inequality, we have
\begin{align*}
\mathcal{B}(\pi;i,j) 
&\leq 
\iint \frac{\eta^2}{(g+1)^{2}}F(g,\eta;w,s)\{h'_i(\eta)\}^2\pi(g)k_j^2(g)\rd g\rd \eta .
\end{align*}
By Lemma \ref{LEM:main},
the integral for $\mathcal{B}$ is 
\begin{equation} \label{product.b}
\begin{split}
 \iint \|x\|^2s^{n/2-1} \mathcal{B}(\pi;i,j) \rd x\rd s 
&\leq q_3(0)\int \eta\{h'_i(\eta)\}^2\rd\eta \int \frac{\pi(g)k^2_j(g)}{g+1}\rd g, 
\end{split}
\end{equation}
where
\begin{align}\label{q3q3}
 q_3(m)=2^{p/2+n/2+1}\pi^{p/2}
 \frac{\Gamma(p/2+1-m)\Gamma(n/2+m)}{\Gamma(p/2)}.
\end{align}

For $\mathcal{C}$, again by Cauchy-Schwarz inequality, we have
\begin{align*}
\mathcal{C}(\pi;i,j)
&\leq 
\frac{1}{w^2}\int  (g+1+w)^2F(g,\eta;w,s)h_i^2(\eta)\pi(g)\{k'_j(g)\}^2\rd g\rd \eta \\
&\leq 
2\int  \left(\frac{(g+1)^2}{w^2}+1\right)F(g,\eta;w,s)h_i^2(\eta)\pi(g)\{k'_j(g)\}^2\rd g\rd \eta.
\end{align*}
By Lemma \ref{LEM:main},
the integral for $\mathcal{C}$ is 
\begin{equation}\label{product.c}
\begin{split}
& \iint \|x\|^2s^{n/2-1} \mathcal{C}(\pi;i,j)\rd x\rd s \\
&\leq 2\left\{q_3(0)+q_3(2)\right\}\int \frac{h_i^2(\eta)}{\eta}\rd\eta \int (g+1)\pi(g)\{k'_j(g)\}^2\rd g . 
\end{split}
\end{equation}
By Lemmas \ref{lem.hi} and \ref{lem:k_j} as well as $\pi(g)\leq 1$ under CASE II, 
the integrals in \eqref{product.b} and \eqref{product.c} are bounded as 
\begin{equation}\label{product.bc.0}
 \begin{split}
 \int \eta\{h'_i(\eta)\}^2\rd\eta \int \frac{\pi(g)k^2_j(g)}{g+1}\rd g< 4\frac{\log(1+j)}{i}, \\
\int \frac{h_i^2(\eta)}{\eta}\rd\eta \int (g+1)\pi(g)\{k'_j(g)\}^2\rd g <10\frac{i}{\log(1+j)}.
\end{split}
\end{equation}
For $\mathcal{D}$, using Lemma \ref{thm:a1a2a3}, we have a following lemma.
\begin{lemma}\label{thm:bcd}
There exist positive constants 
$\tilde{\mathcal{D}}_1$ and $\tilde{\mathcal{D}}_2$ all independent of $i$ and $j$ such that
\begin{align}
 \iint \|x\|^2s^{n/2-1}\mathcal{D}(\pi;i,j)\rd x\rd s \leq \tilde{\mathcal{D}}_1+\frac{i}{\{\log(1+j)\}^2}\tilde{\mathcal{D}}_2.
\end{align}
\end{lemma}
\begin{proof}
See Appendix \ref{sec:thm:bcd}.
\end{proof}

By \eqref{product.b}, \eqref{product.c}, \eqref{product.bc.0} and Lemma \ref{thm:bcd},
with the choice
\begin{equation*}
i=\log(1+j)
\end{equation*}
we have
\begin{align}
\sup_{\substack{j\in \mathbb{N} \\ i=\log(1+j)}}\iint \|x\|^2s^{n/2-1}\left\{4\mathcal{B}(\pi;i,j)+4\mathcal{C}(\pi;i,j)+\mathcal{D}(\pi;i,j)\right\} \rd x\rd s <\infty
\end{align}
and by the dominated convergence theorem,  we have a following result.
\begin{thm}\label{thm:main.case.2}
For $p\geq 3$, the generalized Bayes estimator under $\pi_*$ with 
\begin{align*}
a=-2 \text{ and }b\geq 0
\end{align*}
is admissible among the class of all estimators.
\end{thm}

\appendix
\section{Proof of Lemma \ref{thm:a1a2a3}}
\label{sec:thm:a1a2a3}
\subsection{Proof for $A_1$}
Recall
\begin{align*}
 \mcA_1(\psi)&= 
\left\{
\iint \psi(g)\left|\frac{1}{\iint  F\pi\rd g\rd \eta}-\frac{h_i^2}{\iint  Fh_i^2\pi\rd g\rd \eta}\right|F\pi\rd g\rd \eta \right\}^2 
\iint  F h_i^2\pi \rd g\rd \eta.
\end{align*}
By Cauchy-Schwarz inequality, we have
\begin{align}
\mcA_1(\psi)
&\leq\iint \psi^2(g)F\pi\rd g\rd \eta \iint
\left(\frac{1}{\iint  F\pi\rd g\rd \eta}-\frac{h_i^2}{\iint  Fh_i^2\pi\rd g\rd \eta}\right)^2 F\pi\rd g\rd \eta \notag\\
&\qquad\times\iint  F h_i^2\pi \rd g\rd \eta. \label{EQ.LEM:maru2009.1}
\end{align}
Note
\begin{align}
& \left(\frac{ 1}{\iint  F\pi\rd g\rd \eta}-\frac{ h_i^2}{\iint  Fh_i^2 \pi \rd g\rd \eta}\right)^2 \label{EQ.LEM:maru2009.2}\\
&=\left(\frac{ 1}{ \sqrt{\iint  F\pi \rd g\rd \eta}}-\frac{ h_i}{ \sqrt{\iint  Fh_i^2 \pi \rd g\rd \eta}}\right)^2
\left(\frac{ 1}{ \sqrt{\iint  F\pi \rd g\rd \eta}}+\frac{ h_i}{ \sqrt{\iint  Fh_i^2 \pi \rd g\rd \eta}}\right)^2 \notag\\
&=\frac{1}{\iint  F\pi \rd g\rd \eta\iint  Fh_i^2 \pi \rd g\rd \eta}
\left(1-\frac{h_i\sqrt{\iint  F\pi \rd g\rd \eta}}{\sqrt{\iint  Fh_i^2 \pi \rd g\rd \eta}}\right)^2
\left(\frac{\sqrt{\iint  Fh_i^2 \pi \rd g\rd \eta}}{\sqrt{\iint  F\pi \rd g\rd \eta}}+h_i\right)^2 \notag\\
&\leq \frac{2^2}{\iint  F\pi \rd g\rd \eta\iint  Fh_i^2 \pi \rd g\rd \eta}
\left(1-\frac{h_i\sqrt{\iint  F\pi \rd g\rd \eta}}{\sqrt{\iint  Fh_i^2 \pi \rd g\rd \eta}}\right)^2,\notag
\end{align}
where the inequality follows from the fact $0\leq h_i\leq 1$.

Further we have
\begin{equation}\label{EQ.LEM:maru2009.3}
 \begin{split}
&\iint F \pi \left(1-\frac{h_i\sqrt{\iint  F\pi \rd g\rd \eta}}{\sqrt{\iint  Fh_i^2 \pi \rd g\rd \eta}}\right)^2 \rd g\rd \eta   \\
&=2\iint F \pi \rd g\rd \eta  - 2\frac{\sqrt{\iint  F\pi \rd g\rd \eta}}{\sqrt{\iint  Fh_i^2 \pi \rd g\rd \eta}}\iint Fh_i \pi \rd g\rd \eta  \\
&=2\iint F \pi \rd g\rd \eta \left(1- \sqrt{\frac{(\iint Fh_i \pi \rd g\rd \eta )^2}
{\iint  F\pi \rd g\rd \eta \iint  Fh_i^2 \pi \rd g\rd \eta}}\right)\\
&\leq 2\iint F \pi \rd g\rd \eta \left(1- \frac{(\iint Fh_i \pi \rd g\rd \eta )^2}
{\iint  F\pi \rd g\rd \eta \iint  Fh_i^2 \pi \rd g\rd \eta}\right),
\end{split}
\end{equation}
where the inequality follows from the fact
\begin{align*}
 \frac{(\iint Fh_i \pi \rd g\rd \eta )^2}
{\iint  F\pi \rd g\rd \eta \iint  Fh_i^2 \pi \rd g\rd \eta}\in(0,1),
\end{align*}
which is shown by Cauchy-Schwarz inequality.
By \eqref{EQ.LEM:maru2009.1}, \eqref{EQ.LEM:maru2009.2} and \eqref{EQ.LEM:maru2009.3}, we have
\begin{align}\label{EQ.LEM:maru2009.4}
\mcA_1(\psi)
\leq 8\iint  F\psi^2(g) \pi(g) \rd g\rd \eta\left(1- \frac{(\iint Fh_i \pi \rd g\rd \eta )^2}
{\iint  F\pi \rd g\rd \eta \iint  Fh_i^2 \pi \rd g\rd \eta}\right).
\end{align}
By Lemma \ref{lem.cv}, we have
\begin{align*}
& \iint  F h_i^l \pi \rd g\rd \eta 
= \frac{(1-z)^{p/2+a+1}}{s^{p/2+n/2+1}} \\
&\qquad\times\iint  
\frac{t^{p/2+a}(1-t)^{b}}{(1-zt)^{p/2+a+b+2}}v^{p/2+n/2}\exp\left(-\frac{v}{2(1-zt)}\right)  h_i^l(v/s) \rd t\rd v,
\end{align*}
where $l=0,1,2$. 
For \eqref{EQ.LEM:maru2009.4}, by Lemma 2.8 of \cite{Maruyama-Strawderman-2020-arxiv}, 
there exists a positive constant $q_2(a,b)$, independent of $i$, $x$ and $s$, such that
\begin{equation}\label{EQ.LEM:maru2009.5}
1- \frac{(\iint Fh_i \pi \rd g\rd \eta )^2}
{\iint  F\pi \rd g\rd \eta \iint  Fh_i^2 \pi \rd g\rd \eta}
\leq \frac{q_2(a,b)}{(1+|\log s|)^2}.
\end{equation}
By \eqref{EQ.LEM:maru2009.4}, \eqref{EQ.LEM:maru2009.5} and Lemma \ref{LEM:main.1}, we have
\begin{align*}
&  \iint \frac{\|x\|^2s^{n/2-1}}{(\|x\|^2/s)^m} \mcA_1(\psi;\pi;i,j)\rd x\rd s \\
&\leq 
8q_2(a,b)
\iint \frac{\|x\|^2s^{n/2-1}\psi^2(g) F\pi }{(\|x\|^2/s)^m(1+|\log s|)^2} \rd g\rd \eta \\
&\leq 16 q_3(m) q_2(a,b)B(a+m+2\alpha,b+1),
\end{align*}
where $q_3(m)$ is given by \eqref{q3q3}. This completes the proof.

\subsection{Proof for $A_2$}
Recall
\begin{align*}
\mcA_2(\psi) =
\frac{\dps \left(\iint \psi(g) F h_i^2 \pi (1-k_j^2)\rd g\rd \eta\right)^2}{\iint  Fh_i^2\pi\rd g\rd \eta} .
%
\end{align*}
By 
Cauchy-Schwarz inequality, we have
\begin{equation}\label{A2.0}
 \begin{split}
\mcA_2(\psi)&\leq 
\iint \psi^2(1-k_j^2)^2 F h_i^2 \pi\rd g\rd \eta.
\end{split}
\end{equation}
Note 
\begin{equation}\label{log.inequality}
 \begin{split}
1-k_j^2=(1+k_j)(1-k_j)&\leq 2(1-k_j) =\frac{2\log(g+1)}{\log(g+1+j)},\\
 \log(g+1)&\leq \frac{(g+1)^{\epsilon}}{\epsilon}, \text{ for }\epsilon>0,
\end{split}
\end{equation}
which gives
\begin{equation}\label{inequality.1-kj2}
1-k_j^2\leq \frac{2\log(g+1)}{\log(1+j)}\leq
\frac{2(1+g)^{\epsilon}}{\epsilon\log(1+j)}.
\end{equation}
By  \eqref{A2.0} and \eqref{inequality.1-kj2}, 
\begin{align*}
\mcA_2(\psi)
&\leq \frac{4}{\epsilon^2\{\log(1+j)\}^2}\iint F h_i^2 \psi^2(g+1)^{2\epsilon}\pi \rd g\rd \eta,
\notag 
\end{align*}
where the integral in the right-hand side is integrable when
\begin{align*}
 p/2+a+1+2\alpha-2\epsilon=(p/2+a+1-2\epsilon)+2\alpha>0
\end{align*}
which is guaranteed by the assumptions.

Then, by Lemma \ref{LEM:main} as well as Part \ref{lem.hi.02} of Lemma \ref{lem.hi}, 
the integral for $\mcA_2$ is bounded as follows:
\begin{align*}
& 
\iint \frac{\|x\|^2s^{n/2-1}}{(\|x\|^2/s)^m} \mcA_2(\psi;\pi;i,j)\rd x\rd s  \\ 
&\leq \frac{i}{\{\log(1+j)\}^2}\frac{8q_3(m)}{\epsilon^2}\int (g+1)^{-m+2\epsilon+1}\psi^2(g)\pi(g)\rd g \\
&=
\frac{i}{\{\log(1+j)\}^2}\frac{8q_3(m)}{\epsilon^2}B(a+m+2\alpha-2\epsilon,b+1),
\end{align*}
which completes the proof.
\subsection{Proof for $A_3$}
Recall
\begin{align*}
 \mcA_3(\psi) 
=\frac{\left(\iint \psi(g) Fh_i^2\pi k_j^2\rd g\rd \eta\right)^2}
{(\iint  Fh_i^2\pi\rd g\rd \eta)^2 \iint  Fh_i^2\pi k_j^2\rd g\rd \eta} 
\left( \iint F h_i^2 \pi \rd g\rd \eta - \iint F h_i^2\pi k^2_j \rd g\rd \eta\right)^2.
\end{align*}
By Cauchy-Schwarz inequality,
we have
\begin{equation}\label{EQ:a3.1}
 \begin{split}
\left(\iint  \psi(g) Fh_i^2\pi k_j^2\rd g\rd \eta\right)^2 
&\leq 
\iint  Fh_i^2\pi k_j^2\rd g\rd \eta \iint \psi^2(g)F h_i^2\pi k_j^2 \rd g\rd \eta \\
&\leq 
\iint  Fh_i^2\pi k_j^2\rd g\rd \eta \iint \psi^2(g)F h_i^2\pi  \rd g\rd \eta ,
\end{split}
\end{equation}
where the second inequality follows from $k_j^2\leq 1$. Hence we have
\begin{align*}
& \left( \iint F h_i^2 \pi \rd g\rd \eta - \iint F h_i^2\pi k^2_j \rd g\rd \eta\right)^2 \\
&\leq \iint F h_i^2 \pi \rd g\rd \eta\iint (1-k_j^2)^2F h_i^2 \pi \rd g\rd \eta \\
&\leq \frac{4}{\epsilon^2\{\log(1+j)\}^2}\iint F h_i^2 \pi \rd g\rd \eta\iint (g+1)^{2\epsilon} F h_i^2 \pi \rd g\rd \eta,
\end{align*}
where the second inequality follows from \eqref{inequality.1-kj2}.
Note the integral $\int (g+1)^{2\epsilon} F \pi \rd g$ is integrable when
\begin{align*}
 p/2+a+1-2\epsilon>0
\end{align*}
which is guaranteed by the assumption.
By Part \ref{lem.cv.3} of Lemma \ref{lem.cv}, there exists a positive constant $q_4$ such that
\begin{equation}
\frac{\iint (g+1)^{2\epsilon} F h_i^2 \pi \rd g\rd \eta}{\iint F h_i^2 \pi \rd g\rd \eta} 
\leq q_4 (w+1)^{2\epsilon} 
\end{equation}
which implies that
\begin{equation}
 \mcA_3(\psi)\leq 
\frac{4q_4(w+1)^{2\epsilon}}{\epsilon^2\{\log(1+j)\}^2}
\iint \psi^2(g)F h_i^2\pi k_j^2 \rd g\rd \eta.
\end{equation}
Note, for $0<2\epsilon<1$, we have 
\begin{align*}
 (w+1)^{2\epsilon}\leq w^{2\epsilon}+1.
\end{align*}
Then, by Lemma \ref{LEM:main} as well as Part \ref{lem.hi.02} of Lemma \ref{lem.hi}, we have
\begin{align*}
& 
\iint \frac{\|x\|^2s^{n/2-1}}{(\|x\|^2/s)^m} \mcA_3(\psi;\pi;i,j)\rd x\rd s  \\ 
&\leq \frac{i}{\{\log(1+j)\}^2}\frac{8q_4}{\epsilon^2}
\left(q_3(m-2\epsilon)\int (g+1)^{-m+2\epsilon+1}\psi^2(g)\pi(g)\rd g\right. \\
&\qquad \left. +q_3(m)\int (g+1)^{-m+1}\psi^2(g)\pi(g)\rd g\right) \\
&\leq \frac{i}{\{\log(1+j)\}^2}\frac{8q_4\{q_3(m-2\epsilon)+q_3(m)\}
}{\epsilon^2}
B(a+m+2\alpha-2\epsilon,b+1),
\end{align*}
which completes the proof.

\section{Proof of Lemma \ref{thm:bcd}}
\label{sec:thm:bcd}
\subsection{$D$ for $b=0$}
Recall $\pi(g)\equiv 1$ for $a=-2$ and $b=0$ and
\begin{align*}
 \mathcal{D}(\pi;i,j)&=\left(\frac{w+1}{w}\right)^2\left(
\frac{\int F(0,\eta;w,s)\rd \eta}{\iint  F(g,\eta;w,s)\rd g\rd \eta}- 
\frac{\int F(0,\eta;w,s)h_i^2\rd \eta}{\iint  F(g,\eta;w,s)h_i^2k_j^2\rd g\rd \eta}
\right)^2 \\
&\quad\times \iint  F(g,\eta;w,s)h_i^2k_j^2\rd g\rd \eta.
\end{align*}
By \eqref{eq.4}, we have 
\begin{align}
&\mathcal{D}(\pi;i,j) \label{D.0}\\
&\leq 2\left(1+1/w^2\right)
\left\{ \left(\frac{\int F(0,\eta;w,s)\rd \eta}{\iint  F(g,\eta;w,s)\rd g\rd \eta}- \frac{\int F(0,\eta;w,s)h_i^2\rd \eta}{\iint  F(g,\eta;w,s)h_i^2\rd g\rd \eta}
\right)^2  \right. \notag\\
&\qquad\left. 
+\left(\frac{\int F(0,\eta;w,s)h_i^2\rd \eta}{\iint  F(g,\eta;w,s)h_i^2\rd g\rd \eta}- \frac{\int F(0,\eta;w,s)h_i^2\rd \eta}{\iint  F(g,\eta;w,s)h_i^2k_j^2\rd g\rd \eta}
\right)^2
\right\}\notag\\
&\quad\times \iint  F(g,\eta;w,s)h_i^2 k_j^2 \rd g\rd \eta. \notag
\end{align}
In \eqref{D.0}, we have
\begin{align}
&\left(\frac{\int F(0,\eta;w,s)\rd \eta}{\iint  F(g,\eta;w,s)\rd g\rd \eta}- \frac{\int F(0,\eta;w,s)h_i^2\rd \eta}{\iint  F(g,\eta;w,s)h_i^2\rd g\rd \eta}
\right)^2 \label{db0.1}\\
&\leq  
\left(\int F(0,\eta;w,s)\left|\frac{1}{\iint  F\rd g\rd \eta}-\frac{h_i^2}{\iint  Fh_i^2\rd g\rd \eta}\right|\rd \eta\right)^2\notag
\end{align}
and
\begin{align}
& \left(\frac{\int F(0,\eta;w,s)h_i^2\rd \eta}{\iint  F(g,\eta;w,s)h_i^2\rd g\rd \eta}- \frac{\int F(0,\eta;w,s)h_i^2\rd \eta}{\iint  F(g,\eta;w,s)h_i^2k_j^2\rd g\rd \eta}
\right)^2 \label{db0.2}\\
&= \left(\frac{\int F(0,\eta;w,s)h_i^2\rd \eta\iint  F(g,\eta;w,s)h_i^2(1-k_j^2)\rd g\rd \eta}{\iint  F(g,\eta;w,s)h_i^2\rd g\rd \eta\iint  F(g,\eta;w,s)h_i^2k_j^2\rd g\rd \eta}\right)^2\notag.
\end{align}
Applying the following inequality to \eqref{db0.1}, 
\begin{align}
 F(0,\eta;w,s)
&=\eta^{p/2+n/2}\exp\left(-\frac{\eta s}{2}(w+1)\right) \label{aa1.1} \\
&=\eta^{p/2+n/2}\exp\left(-\frac{\eta s}{2}(w+1)\right)\frac{\int (g+1)^{-p/2-2}\rd g}{\int (g+1)^{-p/2-2}\rd g} \notag \\
&\leq\frac{\eta^{p/2+n/2}\exp(-\eta s(w+1)/2)}{\int (g+1)^{-p/2-2}\rd g} 
\int \frac{1}{(g+1)^{p/2+2}}\exp\left(\frac{\eta s}{2}w\frac{g}{g+1}\right)\rd g \notag \\
&=\frac{\int (g+1)^{-2} F(g,\eta;w,s)\rd g}{\int (g+1)^{-p/2-2}\rd g}.\notag 
\end{align}
and the following inequality to \eqref{db0.2}, 
\begin{align}\label{aa2.0}
 F(0,\eta;w,s)
\leq \frac{\int (g+1)^{-2} F(g,\eta;w,s)\pi(g)k_j^2(g)\rd g}{\int (g+1)^{-p/2-2}k_j^2(g)\rd g},
\end{align}
 we have 
\begin{align*}
\mathcal{D}(\pi;i,j)\leq   \frac{2(1+1/w^2)}{\{\int (g+1)^{-p/2-2}k_1^2(g)\rd g\}^2}
\left\{\mcA_1((g+1)^{-2})+\mcA_3((g+1)^{-2})\right\}.
\end{align*}
In Lemma \ref{thm:a1a2a3}, we set $\alpha=2$, $m=0$ or $2$, $\epsilon=1/8$ as well as $a=-2$ and $b=0$.
There exist positive constants $\tilde{\mathcal{D}}_1$ and $\tilde{\mathcal{D}}_2$, both independent of $i$ and $j$, such that
\begin{align*}
 \iint \|x\|^2s^{n/2-1}\mathcal{D}(\pi;i,j)\rd x\rd s \leq \tilde{\mathcal{D}}_1+\frac{i}{\{\log(1+j)\}^2}\tilde{\mathcal{D}}_2.
\end{align*}

\subsection{$D$ for $b>0$}
When $b>0$, we have
\begin{align*}
&\mathcal{D}(\pi;i,j) \\&=\frac{b^2}{w^2}\left(
\begin{gathered}
\frac{\iint \left(\dfrac{1}{g}+\dfrac{w}{g(g+1)}\right) F\pi\rd g\rd \eta}{\iint  F\pi\rd g\rd \eta}
-\frac{\iint \left(\dfrac{1}{g}+\dfrac{w}{g(g+1)}\right) Fh_i^2\pi k_j^2\rd g\rd \eta}{\iint  Fh_i^2\pi k_j^2\rd g\rd \eta}
\end{gathered}
\right)^2 
\\
&\quad\times 
\iint  F h_i^2 \pi k_j^2 \rd g\rd \eta .
\end{align*}
By \eqref{eq.4} and Lemma \ref{lem:g->g+1} below, we have 
\begin{align*}
\mathcal{D}(\pi;i,j)&
\leq 6b^2\left\{\frac{\mathcal{A}_1(1/g)}{w^2}+\mathcal{A}_1(1/\{g(g+1)\})+
\frac{\mathcal{A}_2(1/g)}{w^2} \right. \\ &\qquad \left. 
+\mathcal{A}_2(1/\{g(g+1)\})+\frac{\mathcal{A}_3(1/g)}{w^2}+\mathcal{A}_3(1/\{g(g+1)\})\right\} \\
&\leq 6b^2r_*^2\left\{\frac{\mathcal{A}_1(1/(g+1))}{w^2}+\mathcal{A}_1(1/(g+1)^2)+
\frac{\mathcal{A}_2(1/(g+1))}{w^2} \right. \\ &\qquad \left. 
+\mathcal{A}_2(1/(g+1)^2)+\frac{\mathcal{A}_3(1/(g+1))}{w^2}+\mathcal{A}_3(1/(g+1)^2)\right\},
\end{align*}
where $r_*=\max(r_1,r_2,r_3) $ defined in Lemma \ref{lem:g->g+1} below.

In Lemma \ref{thm:a1a2a3}, we set $\alpha=2$, $m=0$ or $2$, $\epsilon=1/8$ as well as $a=-2$ and $b>0$.
Then there exist positive constants $\tilde{\mathcal{D}}_1$ and $\tilde{\mathcal{D}}_2$ both independent of $i$ and $j$ such that
\begin{align*}
 \iint \|x\|^2s^{n/2-1}\mathcal{D}(\pi;i,j)\rd x\rd s \leq \tilde{\mathcal{D}}_1+\frac{i}{\{\log(1+j)\}^2}\tilde{\mathcal{D}}_2.
\end{align*}

\begin{lemma}\label{lem:g->g+1}
\begin{align*}
 \int \frac{F\pi\rd g}{g}&\leq r_1\int \frac{F\pi\rd g}{g+1}, &\quad
 \int \frac{F\pi\rd g}{g(g+1)}&\leq r_1\int \frac{F\pi\rd g}{(g+1)^2}, \\
 \int \frac{F\pi(1-k_j^2)\rd g}{g}&\leq r_2\int \frac{F\pi(1-k_j^2)\rd g}{g+1}, &\quad
 \int \frac{F\pi(1-k_j^2)\rd g}{g(g+1)}&\leq r_2\int \frac{F\pi(1-k_j^2)\rd g}{(g+1)^2}, \\
 \int \frac{F\pi k_j^2\rd g}{g}&\leq r_3\int \frac{F\pi k_j^2\rd g}{g+1}, & \quad
 \int \frac{F\pi k_j^2\rd g}{g(g+1)}&\leq r_3\int \frac{F\pi k_j^2\rd g}{(g+1)^2} ,
\end{align*}
where
\begin{align*}
r_1&=\frac{\int \{(g+1)/g\}\pi(g)(g+1)^{-p/2-2}\rd g}{\int \pi(g)(g+1)^{-p/2-2}\rd g}, \\
 r_2&=\max_j\frac{\int \{(g+1)/g\}\pi(g)\{1-k_j^2(g)\}(g+1)^{-p/2-2}\rd g}{\int \pi(g)\{1-k_j^2(g)\}(g+1)^{-p/2-2}\rd g},  \\
r_3&=\max_j\frac{\int \{(g+1)/g\}\pi(g)k_j^2(g)(g+1)^{-p/2-2}\rd g}{\int \pi(g)k_j^2(g)(g+1)^{-p/2-2}\rd g}.
\end{align*}
\end{lemma}
\begin{proof}
We only prove two inequalities in the first line. By the covariance inequality, we have
\begin{align*}
& \int \frac{1}{g}\frac{\pi(g)}{(g+1)^{p/2}}
\exp\left(-\frac{\eta s}{g+1}\frac{w}{2}\right)\rd g \\
& =\int \frac{g+1}{g}\frac{\pi(g)}{(g+1)^{p/2+1}}
\exp\left(-\frac{\eta s}{g+1}\frac{w}{2}\right)\rd g \\
&\leq \frac{\int \{(g+1)/g\}\pi(g)(g+1)^{-p/2-1}\rd g}{\int \pi(g)(g+1)^{-p/2-1}\rd g}
 \int \frac{1}{g+1}\frac{\pi(g)}{(g+1)^{p/2}}
\exp\left(-\frac{\eta s}{g+1}\frac{w}{2}\right)\rd g. 
\end{align*}
where the inequality follows from the fact $ (g+1)/g$ is decreasing and 
$ \exp(-\eta s w/\{2(g+1)\})$ is increasing. Similarly we have
\begin{align*}
 & \int \frac{1}{g(g+1)}\frac{\pi(g)}{(g+1)^{p/2}}
\exp\left(-\frac{\eta s}{g+1}\frac{w}{2}\right)\rd g \\
&\leq \frac{\int \{(g+1)/g\}\pi(g)(g+1)^{-p/2-2}\rd g}{\int \pi(g)(g+1)^{-p/2-2}\rd g} 
 \int \frac{1}{(g+1)^2}\frac{\pi(g)}{(g+1)^{p/2}}
\exp\left(-\frac{\eta s}{g+1}\frac{w}{2}\right)\rd g.
\end{align*}
Since $ (g+1)/g$ is decreasing and $g+1$ is increasing, 
\begin{align*}
 \frac{\int \{(g+1)/g\}\pi(g)(g+1)^{-p/2-1}\rd g}{\int \pi(g)(g+1)^{-p/2-1}\rd g} 
& =\frac{\int \{(g+1)/g\}(g+1)\pi(g)(g+1)^{-p/2-2}\rd g}{\int (g+1)\pi(g)(g+1)^{-p/2-2}\rd g} \\
&<\frac{\int \{(g+1)/g\}\pi(g)(g+1)^{-p/2-2}\rd g}{\int \pi(g)(g+1)^{-p/2-2}\rd g} \\
&=r_1,
\end{align*}
which completes the proof of two inequalities in the first line.
\end{proof}

\section{Lemmas and Proofs}
\subsection{the sequence}
\label{sec:seq}
\begin{lemma}\label{lem.hi}
Let
\begin{equation}
 h_i(\eta)=\frac{i}{i+|\log\eta|}.
\end{equation}
\begin{enumerate}
\item 
\label{lem.hi.01}$h_i(\eta)$ is increasing in $i$ and 
\begin{align*}
 \lim_{i\to\infty}h_i(\eta)=1 \text{ for all }\eta>0.
\end{align*}
\item \label{lem.hi.02} $ \dps\int_0^\infty \eta^{-1}h_i^2(\eta)\rd \eta=2i$.
\item \label{lem.hi.03} $\dps \int_0^\infty \eta \{h'_i(\eta)\}^2\rd \eta\leq \frac{2}{i}$.
\end{enumerate}
\end{lemma}
\begin{proof}\mbox{}
 [Part \ref{lem.hi.01}] This part is straightforward given the form of $h_i( \eta )$.

[Part \ref{lem.hi.02}] The results follow from the integrals,
\begin{equation}\label{2i2i}
 \begin{split}
 \int_0^\infty \eta^{-1}h_i^2(\eta)\rd \eta&=
 \int_0^1 \frac{i^2\rd\eta}{\eta \{i+\log (1/\eta)\}^2}
+\int_1^\infty \frac{i^2\rd\eta}{\eta \{i+\log \eta\}^2}  \\
&= \left[\frac{i^2}{i+\log (1/\eta)}\right]_0^1+\left[-\frac{i^2}{i+\log \eta}\right]_1^\infty=
 2i.
\end{split}
\end{equation}
[Part \ref{lem.hi.03}] Note
\begin{align*}
 h'_i(\eta)=
\begin{cases}
\dps \frac{i}{\eta\{i+\log (1/\eta)\}^2} & 0<\eta<1, \\
\dps -\frac{i}{\eta\{i+\log \eta\}^2} & \eta\geq 1.
\end{cases}
\end{align*}
 Hence
\begin{align*}
 \{h'_i(\eta)\}^2
&=\begin{cases}
\dps \frac{i^2}{\eta^2\{i+\log (1/\eta)\}^4} & 0<\eta<1 \\
\dps -\frac{i^2}{\eta^2\{i+\log \eta\}^4} & \eta\geq 1
\end{cases} \\
&\leq 
\begin{cases}
\dps \frac{1}{\eta^2\{i+\log (1/\eta)\}^2} & 0<\eta<1, \\
\dps \frac{1}{\eta^2\{i+\log \eta\}^2} &  \eta\geq 1.
\end{cases}
\end{align*}
Then
\begin{align*}
 \int_0^\infty \eta \{h'_i(\eta)\}^2\rd \eta\leq 
 \int_0^1 \frac{\rd\eta}{\eta \{i+\log (1/\eta)\}^2}
+\int_1^\infty \frac{\rd\eta}{\eta \{i+\log \eta\}^2}=\frac{2}{i}.
\end{align*}
\end{proof}

\begin{lemma}\label{lem:k_j}
Let
\begin{equation}
 k_j( g )=1-\frac{\log( g +1)}{\log( g +1+j)}.
\end{equation}
  \begin{enumerate}
   \item \label{lem:k_j.1}
$k_j( g )$ is increasing in $j$ for fixed $ g $, and decreasing in $ g $ for fixed $j$. Further
$\lim_{j\to\infty}k_j( g )=1$ for fixed $ g \geq 0$.
   \item \label{lem:k_j.2}
	 For fixed $j\geq 1$,
\begin{align*}
k_j( g )\leq
 \frac{(1+j)\log(1+j)}{( g +1+j) \log( g +1+j)}.
\end{align*}
   \item \label{lem:k_j.2.5} Let $\pi(g)=(g+1)^{-a-2}\{g/(g+1)\}^b$ for $a\geq -2$ and $b>-1$.
Then\begin{align*}
 \int_0^\infty \pi(g)k_j^2(g)\rd g\leq \frac{1}{b+1}+\max(1,2^{-b})(1+j).
\end{align*}
   \item \label{lem:k_j.3}
\begin{align*}
 \int_0^\infty \frac{k_j^2(g)}{g+1}\rd g\leq 2\log (1+j).
\end{align*}
   \item \label{lem:k_j.4}
	 For $ g \geq 0$,
\begin{align*}
 k'_j( g )=-\frac{j/(g+1)+k_j( g )}{( g +1+j) \log( g +1+j)}
\end{align*}
   \item \label{lem:k_j.5}
\begin{align*}
 \int_0^\infty (g+1)\{k'_j(g)\}^2\rd g\leq \frac{5}{\log (1+j)}.
\end{align*}
\end{enumerate}
 \end{lemma}
\begin{proof}\mbox{}
 [Part \ref{lem:k_j.1}] This part is straightforward given the form of $k_j( g )$.

 [Part \ref{lem:k_j.2}] The function $k_j( g )$ is rewritten as
 \begin{align*}
 k_j( g )
=\frac{j\zeta(j/( g +1+j))}{( g +1+j) \log( g +1+j)},
 \end{align*}
where $\zeta(x)=-\log(1-x)/x=1+\sum_{l=1}^\infty x^l/(l+1)$ which is increasing in $x$.
 Hence
 \begin{align*}
  k_j( g )\leq
  \frac{j\zeta(j/(1+j))}{( g +1+j) \log( g +1+j)}=\frac{(1+j)\log(1+j)}{( g +1+j) \log( g +1+j)}.
 \end{align*}

[Part \ref{lem:k_j.2.5}]
By Part \ref{lem:k_j.2},
\begin{align*}
 \int_0^\infty k_j^2( g )\rd  g  
&\leq 
 \int_0^\infty 
\frac{(1+j)^2\{\log(1+j)\}^2}{( g +1+j)^2 \{\log( g +1+j)\}^2}\rd  g  \\
&\leq  \int_0^\infty 
\frac{(1+j)^2}{( g +1+j)^2}\rd  g  \\
&=  1+j.
\end{align*}
Under the condition, we have $\pi(g)\leq \{g/(g+1)\}^b$. 
When $b\geq 0$, we have $\pi(g)\leq 1$ and hence
\begin{align*}
 \int_0^\infty \pi(g)k_j^2( g )\rd  g  \leq \int_0^\infty k_j^2( g )\rd  g  \leq 1+j.
\end{align*}
When $-1<b<0$, 
\begin{align*}
\left(\frac{g}{g+1}\right)^b&\leq \left(1+ g^b\right) I_{(0,1)}(g)+2^{-b}I_{(1,\infty)}(g) \\
&\leq g^b I_{(0,1)}(g) +2^{-b}I_{(0,\infty)}(g). 
\end{align*}
Note
\begin{align*}
  \int_0^1 \pi(g)k_j^2( g )\rd  g\leq \int_0^1 g^{b}\rd g=\frac{1}{b+1}.
\end{align*} 
Then the result follows.

[Part \ref{lem:k_j.3}]
Note $k_j^2\leq 1$  by definition. Hence
\begin{align*}
 \int_0^j \frac{k_j^2( g )}{ g +1}\rd  g \leq 
\int_0^j \frac{1}{ g +1}\rd  g =\log( 1+j).
\end{align*}
Also, by Part \ref{lem:k_j.2},
\begin{align*}
 \int_j^\infty \frac{k_j^2( g )}{( g +1)}\rd  g  
&\leq 
 \int_j^\infty 
\frac{(1+j)^2\{\log(1+j)\}^2}{( g +1)( g +1+j)^2 \{\log( g +1+j)\}^2}\rd  g  \\
&\leq  \int_j^\infty 
\frac{\{\log(1+j)\}^2}{( g +1+j) \{\log( g +1+j)\}^2}\rd  g  \\
&=  \frac{\{\log(1+j)\}^2}{\log(1+2j)} \\
&\leq \log(1+j).
\end{align*}
Then the result follows.

[Part \ref{lem:k_j.4}] The derivative is 
 \begin{equation}\label{h_deri}
  k'_j( g )=-\frac{1}{( g +1)\log( g +1+j)} 
   +\frac{\log( g +1)}{( g +1+j)\{\log( g +1+j)\}^2}.
 \end{equation}
 Then
\begin{align*}
\log( g +1+j) k'_j( g ) 
 &=-\frac{1}{ g +1}+\frac{\log( g +1)}{( g +1+j)\log( g +1+j)}\\
&=-\frac{1}{ g +1}+\frac{1}{ g +1+j}\left\{1-k_j( g )\right\}\\
&=-\frac{j}{( g +1)( g +1+j)}-\frac{k_j( g )}{ g +1+j}.
\end{align*}

[Part \ref{lem:k_j.5}]
By Part \ref{lem:k_j.4}, we have
\begin{align*}
&(g+1) \{k'_j( g )\}^2 \\ &\leq 2\left( \frac{j^2}{(g+1)(g +1+j)^2\{\log(g +1+j)\}^2 }
+ \frac{(g+1)k_j^2(g)}{(g +1+j)^2\{\log(g +1+j)\}^2 }
\right).
\end{align*}
Then we have
\begin{align*}
\int_0^j \frac{j^2 \ \rd g}{(g+1)(g +1+j)^2\{\log(g +1+j)\}^2 }
&\leq \frac{1}{\{\log(1+j)\}^2}\int_0^j\frac{\rd g}{g+1} \\ &=\frac{1}{\log(1+j)}
\end{align*}
and
\begin{align*}
\int_j^\infty \frac{j^2 \ \rd g}{(g+1)(g +1+j)^2\{\log(g +1+j)\}^2 }
&\leq \frac{1}{4}\int_j^\infty \frac{\rd g}{(g +1)\{\log(g +1)\}^2 } \\ &=\frac{1}{4\log(1+j)}.
\end{align*}
Further, by $0\leq k_j\leq 1$,
\begin{align*}
\int_0^\infty \frac{(g+1)k_j^2(g) \ \rd g}{(g +1+j)^2\{\log(g +1+j)\}^2 }
&\leq \int_0^\infty \frac{\rd g}{(g +1+j)\{\log(g +1+j)\}^2 }\\ &=\frac{1}{\log (1+j)}.
\end{align*}
Then the result follows.
\end{proof}

\subsection{Proof of Lemma \ref{LEM:integral.parts.a=2}}
\label{sec:LEM:integral.parts.a=2}
Lemma \ref{LEM:integral.parts.a=2} follows from the following result.
\begin{lemma}\label{LEM:integral.parts} 
Assume $a=-2$ and $b\geq 0$.
Then 
\begin{align*}
&(n/2+1)w \iint (g+1)^{-1}F(g,\eta;w,s)h_i^2(\eta)\pi(g)k_j^2(g)\rd g\rd \eta \\
& = (p/2-1)\iint  F(g,\eta;w,s)h_i^2(\eta)\pi(g)k_j^2(g)\rd g\rd \eta \\
&\quad - 2w\iint \frac{\eta}{g+1} F(g,\eta;w,s)h_i(\eta)h'_i(\eta)\pi(g)k_j^2(g)\rd g\rd \eta \\
&\quad - 2\iint (g+1+w)F(g,\eta;w,s)h_i^2(\eta)\pi(g)k_j(g)k'_j(g)\rd g\rd \eta \\
&\quad - 
\begin{cases}
\displaystyle b\iint \frac{g+1+w}{g(g+1)}F(g,\eta;w,s)h_i^2(\eta)\pi(g)k_j^2(g)\rd g\rd \eta &b>0, \\
\displaystyle\int (1+w)F(0,\eta;w,s)h_i^2(\eta)\rd \eta & b=0.
\end{cases}
\end{align*}
\end{lemma}
\begin{proof}
By change of variables $ u=\eta \{1+w/(g+1)\}$ with $\rd \eta/\rd u=1/\{1+w/(g+1)\}$, we have
\begin{align*}
&\iint (g+1)^{-1}F(g,\eta;w,s)h_i^2(\eta)\pi(g)k_j^2(g)\rd g\rd \eta \\
&=\iint \left(1+\frac{w}{g+1}\right)^{-p/2-n/2-1}\frac{u^{p/2+n/2}}{(g+1)^{p/2+1}}
\exp\left(-\frac{su}{2}\right) \\
&\quad \times h_i^2\left(\frac{u}{1+w/(g+1)}\right)\pi(g)k_j^2(g)\rd g\rd u \\
&=
\int \frac{(g+1)^{n/2}}{(g+1+w)^{p/2+n/2+1}}\zeta_i(w/(g+1),s) \left(\frac{g}{g+1}\right)^bk_j^2(g)\rd g,
\end{align*}
where
\begin{align*}
 \zeta_i(v,s)=\int u^{p/2+n/2}\exp\left(-\frac{su}{2}\right)h_i^2\left(\frac{u}{1+v}\right)\rd u.
\end{align*}
Note
\begin{align*}
&\frac{ w(g+1)^{n/2}}{(g+1+w)^{p/2+n/2+1}} \\
&=w\left(\frac{g+1}{g+1+w}\right)^{n/2}(g+1+w)^{-p/2-1} \\
&=w \left(1-\frac{w}{g+1+w}\right)^{n/2}(g+1+w)^{-p/2-1} \\
&=\frac{\rd }{\rd g}\left\{\frac{1}{n/2+1} \left(1-\frac{w}{g+1+w}\right)^{n/2+1}\right\}(g+1+w)^{-p/2+1}.
\end{align*}
Then an integration by parts gives
\begin{align}
& w\int \frac{(g+1)^{n/2}}{(g+1+w)^{p/2+n/2+1}}\zeta_i(w/(g+1),s) \left(\frac{g}{g+1}\right)^bk_j^2(g)\rd g \label{integ.parts.0} \\
&=\frac{1}{n/2+1}\left\{
\left[\left(1-\frac{w}{g+1+w}\right)^{n/2+1}\frac{\zeta_i(w/(g+1),s)}{(g+1+w)^{p/2-1}}\left(\frac{g}{g+1}\right)^bk^2_j(g)\right]_0^\infty \right. \notag \\
&\quad +
(p/2-1)\int_0^\infty \left(1-\frac{w}{g+1+w}\right)^{n/2+1}\frac{\zeta_i(w/(g+1),s)}{(g+1+w)^{p/2}} 
\left(\frac{g}{g+1}\right)^bk_j^2(g)\rd g
\notag \\
&\quad  - \int_0^\infty  \left(1-\frac{w}{g+1+w}\right)^{n/2+1}\frac{\zeta'_i(w/(g+1),s)}
{(g+1+w)^{p/2-1}}\left\{\frac{-w}{(g+1)^2}\right\} \left(\frac{g}{g+1}\right)^bk_j^2(g)\rd g \notag\\
&\quad  -
\int_0^\infty  \left(1-\frac{w}{g+1+w}\right)^{n/2+1}\frac{\zeta_i(w/(g+1),s)}{(g+1+w)^{p/2-1}}
\left\{\frac{\rd}{\rd g}\left(\frac{g}{g+1}\right)^b\right\}
k^2_j(g)\rd g\notag
\\
&\quad \left. - 2
\int_0^\infty  \left(1-\frac{w}{g+1+w}\right)^{n/2+1}\frac{\zeta_i(w/(g+1),s)}{(g+1+w)^{p/2-1}}
\left(\frac{g}{g+1}\right)^b k_j(g)k'_j(g)\rd g
 \right\}, \notag
\end{align}
where $\pi(0)=1$ for $b=0$, $\pi(0)=0$ for $b>0$, and
\begin{align*}
\frac{\rd}{\rd g}\left(\frac{g}{g+1}\right)^b=b\left(\frac{g}{g+1}\right)^{b-1}\frac{1}{(g+1)^2}=\frac{b}{g(g+1)}\left(\frac{g}{g+1}\right)^b,
\end{align*}
for $b>0$. Further we have
\begin{align*}
 \zeta'_i(v,s) &=\frac{\partial }{\partial v}\zeta_i(v,s) \\
&=-\frac{2}{(v+1)^2}\int u^{p/2+n/2+1}\exp\left(-\frac{su}{2}\right)h_i\left(\frac{u}{1+v}\right)h'_i\left(\frac{u}{1+v}\right)\rd u \\
&=-2(v+1)^{p/2+n/2}\int \eta^{p/2+n/2+1}\exp\left(-\frac{s(1+v)\eta}{2}\right)h_i(\eta)h'_i(\eta)\rd \eta .
\end{align*}
In \eqref{integ.parts.0}, note
\begin{align*}
& \left(1-\frac{w}{g+1+w}\right)^{n/2+1}\frac{1}{(g+1+w)^{p/2}}\left(\frac{g}{g+1}\right)^b \\
&=\left(1+\frac{w}{g+1}\right)^{-p/2-n/2-1}\frac{\pi(g)}{(g+1)^{p/2}}.
\end{align*}
Change of variables $ \eta=u/ \{1+w/(g+1)\}$ with $\rd u/\rd \eta=1+w/(g+1)$
completes the proof.
\end{proof}

\subsection{Lemmas for Lemma \ref{thm:a1a2a3}}
\begin{lemma}\label{LEM:main}
Assume $-n/2<m<p/2+1$ and that both $H(\eta)/\eta$ and $\Pi(g)/(g+1)^{m-1}$ are integrable. Then
\begin{align*}
& \iiiint \frac{\|x\|^2s^{n/2-1}}{(\|x\|^2/s)^m}
F(g,\eta;\|x\|^2/s,s) H(\eta)\Pi(g)\rd g\rd \eta\rd x\rd s\\
&=q_3(m)\int\frac{H(\eta)}{\eta}\rd \eta
\int \frac{\Pi(g)}{(g+1)^{m-1}}\rd g,
\end{align*} 
where
\begin{align*}
 q_3(m)=2^{p/2+n/2+1}\pi^{p/2}
 \frac{\Gamma(p/2+1-m)\Gamma(n/2+m)}{\Gamma(p/2)}.
\end{align*}
\end{lemma}
\begin{proof}
\begin{align*}
&\iiiint \frac{\|x\|^2s^{n/2-1}}{(\|x\|^2/s)^m}
F(g,\eta;\|x\|^2/s,s) 
H(\eta) \Pi(g)\rd g \rd \eta\rd x\rd s \\
&= \iiiint \frac{\|x\|^2s^{n/2-1}}{(\|x\|^2/s)^m}
 \frac{\eta^{p/2+n/2}}{(g+1)^{p/2}}
\exp\left(-\frac{s\eta}{2}\left(\frac{\|x\|^2/s}{1+g}+1\right)\right)H(\eta) \Pi(g)
\rd g \rd \eta\rd x\rd s \\
& =\iiiint (g+1)^{p/2}s^{p/2}
\frac{\|y\|^2 s(g+1) s^{n/2-1}}{\|y\|^{2m}(1+g)^m }
\\ &\qquad\times
\frac{\eta^{p/2+n/2}}{(g+1)^{p/2}}\exp\left(-\frac{s\eta}{2}\left(\|y\|^2+1\right)\right)H(\eta) \Pi(g)
\rd g \rd \eta
\rd y\rd s \\
& =\frac{\Gamma(p/2+n/2+1)}{2^{-p/2-n/2-1}}
\int
\frac{\|y\|^{2(1-m)}\rd y}{(1+\|y\|^2)^{p/2+n/2+1}}
\int\frac{H(\eta)}{\eta}\rd \eta
\int \frac{\Pi(g)\rd g}{(g+1)^{m-1}} \\
& =\frac{\Gamma(p/2+n/2+1)}{2^{-p/2-n/2-1}}
 \frac{\pi^{p/2}\Gamma(p/2+1-m)\Gamma(n/2+m)}{\Gamma(p/2)\Gamma(p/2+n/2+1)}
\int\frac{H(\eta)}{\eta}\rd \eta
\int \frac{\Pi(g)\rd g}{(g+1)^{m-1}},
\end{align*}
where the second equality follows from change of variables $ y_i=x_i/(\sqrt{1+g}\sqrt{s})$ with Jacobian 
$ \left|\partial x/\partial y\right|=(1+g)^{p/2}s^{p/2}$ as well as
$w/(1+g)=\|y\|^2$ and $\|x\|^2=(1+g)s\|y\|^2$, and the last equality follows from Lemma \ref{lem.gamma.beta}.
\end{proof}

\begin{lemma}\label{LEM:main.1}
Assume $-n/2<m<p/2+1$ and $\Pi(g)/(g+1)^{m-1}$ is integrable. Then
\begin{align*}
& \iiiint \frac{\|x\|^2s^{n/2-1}}{(\|x\|^2/s)^m}
\frac{F(g,\eta;\|x\|^2/s,s) \Pi(g)}{(1+|\log s|)^2}
\rd g\rd \eta\rd x\rd s\\
&=2q_3(m)\int \frac{\Pi(g)\rd g}{(g+1)^{m-1}}.
\end{align*} 
\end{lemma}

\begin{proof}
\begin{align*}
& \iint \frac{\|x\|^2s^{n/2-1}}{(\|x\|^2/s)^m}
\frac{F(g,\eta;\|x\|^2/s,s) \Pi(g)}{(1+|\log s|)^2}
\rd g\rd \eta\rd x\rd s \\
&=
\iiiint \frac{\|x\|^2s^{n/2-1}}{(\|x\|^2/s)^m}
\frac{\eta^{p/2+n/2}}{(g+1)^{p/2}}\exp\left(-\frac{\eta s}{2}
\left\{\frac{\|x\|^2/s}{g+1}+1\right\}\right) 
\frac{\Pi(g)\rd g\rd \eta\rd x\rd s}{(1+ |\log s|)^2} \\
&= 
\iiiint  (1+g)^{p/2}s^{p/2}\frac{s^{n/2-1}\|y\|^2(1+g)s}{ \{(1+g)\|y\|^2\}^m}
 \\
&\quad \times 
\frac{\eta^{p/2+n/2}}{(g+1)^{p/2}}\exp\left(-\frac{\eta s}{2}
\left\{\|y\|^2+1\right\}\right)
\frac{ \Pi(g)}{(1+|\log s|)^2} \rd g\rd \eta\rd y\rd s\\
&= \frac{\Gamma(p/2+n/2+1)}{2^{-p/2-n/2-1}}  \int\frac{\rd s}{s(1+|\log s|)^2} 
\int \frac{\|y\|^{2(1-m)}\rd y}{(1+\|y\|^2)^{p/2+n/2+1}}\int \frac{\Pi(g)\rd g}{(g+1)^{m-1}} \\
&=2q_3(m)\int \frac{\Pi(g)\rd g}{(g+1)^{m-1}}.
\end{align*}
\end{proof}

\begin{lemma}\label{lem.gamma.beta}
 Let $y\in\mathbb{R}^p$. Assume $\beta>\alpha>-p/2$. Then
 \begin{align}
  \int_{\mathbb{R}^p}(\|y\|^2)^\alpha(1+\|y\|^2)^{-p/2-\beta}\rd y
=\pi^{p/2}\frac{\Gamma(p/2+\alpha)\Gamma(\beta-\alpha)}{\Gamma(p/2)\Gamma(p/2+\beta)}.
 \end{align}
\end{lemma}

\begin{proof}
 \begin{align*}
  \int_{\mathbb{R}^p}(\|y\|^2)^\alpha(1+\|y\|^2)^{-p/2-\beta}\rd y 
&=\frac{\pi^{p/2}}{\Gamma(p/2)}\int_0^\infty \frac{u^{p/2-1+\alpha}}{(1+u)^{p/2+\beta}}\rd u\\
&=\frac{\pi^{p/2}}{\Gamma(p/2)}\mathrm{Be}(p/2+\alpha, \beta-\alpha) \\
&=\pi^{p/2}\frac{\Gamma(p/2+\alpha)\Gamma(\beta-\alpha)}{\Gamma(p/2)\Gamma(p/2+\beta)}.
 \end{align*}
\end{proof}

\begin{lemma}\label{lem.cv}
Assume $p/2+a+1>0$ and $ 0\leq \gamma<p/2+a+1$. 
\begin{enumerate}
 \item \label{lem.cv.1}
\begin{align*}
 \int_0^\infty \int_0^\infty  F h_i^2(\eta) (g+1)^\gamma  \pi(g) \rd g\rd \eta 
=\frac{(1-z)^{p/2-\gamma+a+1}}{s^{p/2+n/2+1}} \mathcal{H}(z,s;i),
\end{align*}
where $z=w/(1+w)$ and
\begin{align}
& \mathcal{H}(z,s;i)\label{mcH.1} \\
&=
\int_0^1\int_0^\infty  
\frac{t^{p/2-\gamma+a}(1-t)^b}{(1-zt)^{p/2-\gamma+a+b+2}}
v^{p/2+n/2}\exp\left(-\frac{v}{2(1-zt)}\right)  h_i^2(v/s) \rd t\rd v.\notag
\end{align}
\item \label{lem.cv.2}
The function $\{(i+|\log s|)/i\}^2 \mathcal{H}(z,s;i)$ is bounded from above and below, where
the lower and upper bounds are independent of $i$, $z$ and $s$.
\item \label{lem.cv.3}
There exists a positive constant $q_4$, independent of $i$, $z$ and $s$, such that
\begin{equation}
\frac{\int_0^\infty \int_0^\infty  F h_i^2(\eta) (g+1)^\gamma  \pi(g) \rd g\rd \eta }{\int_0^\infty \int_0^\infty  F h_i^2(\eta)  \pi(g) \rd g\rd \eta }
\leq q_4 (w+1)^{\gamma}. 
\end{equation}
\end{enumerate}
\end{lemma}

\begin{proof}
\ [Part \ref{lem.cv.1}] Note
\begin{align*}
& \iint  F (g+1)^\gamma  h_i^2\pi \rd g\rd \eta \\
& =\iint  
\frac{\eta^{p/2+n/2}}{(g+1)^{p/2}}\exp\left\{-\frac{\eta s}{2} 
\left(\frac{w}{g+1}+1\right)\right\} h_i^2(\eta)
(g+1)^{\gamma-a-2} \left(\frac{g}{g+1}\right)^{b}
 \rd g\rd \eta.
\end{align*}
Apply the change of variables 
\begin{align*}
 g=\frac{1-t}{(1-z)t} \text{ where } z=\frac{w}{w+1}
\end{align*}
with
\begin{align*}
 g+1=\frac{1-zt}{(1-z)t}, \quad 1+\frac{w}{g+1}=1+\frac{z/(1-z)}{g+1}=\frac{1}{1-zt},\quad \left|\frac{\rd g}{\rd t}\right|=\frac{1}{(1-z)t^2}.
\end{align*}
Then
\begin{align*}
&=\iint \left(\frac{(1-z)t}{1-zt}\right)^{p/2-\gamma+a+2}
\left(\frac{1-t}{1-zt}\right)^{b} \frac{\eta^{p/2+n/2}}{(1-z)t^2}
\exp\left(-\frac{\eta s}{2(1-zt)}\right)  h_i^2(\eta) \rd t\rd \eta \\
 &=(1-z)^{p/2-\gamma+a+1}\iint  
\frac{t^{p/2-\gamma+a}(1-t)^{b}}{(1-zt)^{p/2-\gamma+a+b+2}}
\eta^{p/2+n/2}\exp\left(-\frac{\eta s}{2(1-zt)}\right)  h_i^2(\eta) \rd t\rd \eta \\
 &=\frac{(1-z)^{p/2-\gamma+a+1}}{s^{p/2+n/2+1}}\iint  
\frac{t^{p/2-\gamma+a}(1-t)^{b}}{(1-zt)^{p/2-\gamma+a+b+2}}
v^{p/2+n/2}\exp\left(-\frac{v}{2(1-zt)}\right)  h_i^2(v/s) \rd t\rd v.
\end{align*}

[Part \ref{lem.cv.2}] Let the probability density investigated in \cite{Maruyama-Strawderman-2020-arxiv} be
\begin{align*}
 f(v\mymid z)&=\frac{v^{(p+n)/2}}{\psi(z)}
\int_0^1 \frac{t^{p/2+a-\gamma}(1-t)^b}{(1-zt)^{p/2-\gamma+a+b+2}}\exp\left(-\frac{v}{2(1-zt)}\right) \rd t,\label{fff}
\end{align*} 
with normalizing constant $\psi(z)$ given by
\begin{align*}
  \psi(z) = \int_0^\infty\int_0^1 \frac{t^{p/2+a-\gamma}(1-t)^b}{(1-zt)^{p/2-\gamma+a+b+2}}v^{(p+n)/2}\exp\left(-\frac{v}{2(1-zt)}\right)\rd v \rd t.
\end{align*}
Then
\begin{align*}
 \mathcal{H}(z,s;i)
=  \psi(z)E\left[h_i^2(V/s)\mymid z\right].  
\end{align*}

By Lemma A.4 of \cite{Maruyama-Strawderman-2020-arxiv}, $0<\psi(z)<\infty$ for $z\in[0,1]$. Hence
we focus on
\begin{align*}
 \left(\frac{i+|\log s|}{i}\right)^2 
E\left[h_i^2(V/s)\mymid z\right].
\end{align*}  

$\bm{\langle}\!\bm{\langle}$lower bound$\bm{\rangle}\!\bm{\rangle}$ 
Note
\begin{equation}\label{hi.hyouka.1}
\begin{split}
 \frac{i+|\log s|}{i+|\log v/s|} 
\geq \frac{i+|\log s|}{i+|\log s|+|\log v|}\geq \frac{1}{1+|\log v|}.
\end{split}
\end{equation}
By the Jensen inequality, we have
\begin{align*}
  \left(\frac{i+|\log s|}{i}\right)^2 
E\left[h_i^2(V/s)\mymid z\right]\geq \left(\frac{1}{1+E[|\log V||z]}\right)^2
\geq \left(\frac{1}{1+\max_z E[|\log V||z]}\right)^2.
\end{align*}

\smallskip

$\bm{\langle}\!\bm{\langle}$$0<s<1$, upper bound$\bm{\rangle}\!\bm{\rangle}$ Assume $0<s<1$.
When $v\geq s$, we have
\begin{equation}\label{hi.hyouka.3}
 \begin{split}
 \left(\frac{i+\log(1/s)}{i+\log v+\log(1/s)}\right)^2 
 &= \left(1-\frac{\log v}{i+\log v+\log(1/s)}\right)^2 \\
&\leq 2+2\frac{|\log v|^2}{\{i+\log v+\log(1/s)\}^2} \\
&\leq 2+2|\log v|^2.
\end{split}
\end{equation}
When $v<s$, we have
\begin{align*}
 \left(\frac{i+\log(1/s)}{i+\log (s/v)}\right)^2\leq
 \left(\frac{i+\log(1/s)}{i}\right)^2 
\leq\{1+\log(1/s)\}^2.
\end{align*}
Then
\begin{align*}
 &\left(\frac{i+\log(1/s)}{i}\right)^2 E\left[h_i^2(V/s)\mymid z\right]  \\
 &=\int_0^s \left(\frac{i+\log(1/s)}{i+\log (s/v)}\right)^2 f(v\mymid z)\rd v
 +\int_s^\infty \left(\frac{i+\log(1/s)}{i+\log (v/s)}\right)^2f(v\mymid z)\rd v \\
 &=\int_0^s \{1+\log(1/s)\}^2 f(v\mymid z)\rd v
 +\int_s^\infty \{2+2|\log v|^2\}f(v\mymid z)\rd v \\
 &\leq \max_{z\in[0,1]}\left( \{1+\log(1/s)\}^2\int_0^s  f(v\mymid z)\rd v\right)
 +2+2\max_{z\in[0,1]}E[|\log V|^2\mymid z].
\end{align*}
As in Lemma A.2 of \cite{Maruyama-Strawderman-2020-arxiv}, $f(v\mymid z)$ is regarded as a generalization of Gamma distribution. 
We have $f(0\mymid z)=0$ and $f(v\mymid z)$ grows with polynomial order around zero. Hence
\begin{align*}
 \sup_{z\in[0,1],s\in(0,1)}\left(\{1+\log(1/s)\}^2\int_0^s  f(v\mymid z)\rd v\right)<\infty.
\end{align*}

\medskip

$\bm{\langle}\!\bm{\langle}$$s>1$, upper bound$\bm{\rangle}\!\bm{\rangle}$
Assume $s>1$ and  $v\leq s$.
As in \eqref{hi.hyouka.3}, we have
\begin{align*}
 \left(\frac{i+\log s}{i+\log (s/v)}\right)^2 
 &= \left(1-\frac{\log (1/v)}{i+\log (s/v)}\right)^2 \\
&\leq 2+2\frac{|\log v|^2}{\{i+\log (s/v)\}^2} \\
&\leq 2+2|\log v|^2.
\end{align*}
When $v>s$,
\begin{align*}
 \left(\frac{i+\log s}{i+\log (v/s)}\right)^2 
\leq  \left(\frac{i+\log s}{i}\right)^2 \leq(1+\log s)^2.
\end{align*}
Then we have
\begin{align*}
 &\left(\frac{i+\log s}{i}\right)^2 E\left[h_i^2(V/s)\mymid z\right]  \\
 &=\int_0^s \left(\frac{i+\log s}{i+\log (s/v)}\right)^2 f(v\mymid z)\rd v
 +\int_s^\infty \left(\frac{i+\log s}{i+\log (v/s)}\right)^2f(v\mymid z)\rd v \\
 &\leq \int_0^s \left(2+2|\log v|^2\right) f(v\mymid z)\rd v
+\int_s^\infty (1+\log s)^2 f(v\mymid z)\rd v \\
 &\leq 2+2\max_{z\in[0,1]}E[|\log V|^2\mymid z]+ \max_{z\in[0,1]}\left((1+\log s)^2\int_s^\infty  f(v\mymid z)\rd v\right).
\end{align*}
As in Lemma A.2 of \cite{Maruyama-Strawderman-2020-arxiv}, $f(v\mymid z)$ is regarded as a generalization of Gamma distribution. 
We have $f(v\mymid z)$ is with exponential decay at infinity. Hence
\begin{align*}
 \max_{z\in[0,1],s\in(1,\infty)}\left((1+\log s)^2\int_s^\infty  f(v\mymid z)\rd v\right)<\infty.
\end{align*}

[Part \ref{lem.cv.3}] Part \ref{lem.cv.3} follows from Parts \ref{lem.cv.1} and \ref{lem.cv.2}.
\end{proof}


\end{document}